\documentclass[12pt]{amsart}

\usepackage{amsmath}
\usepackage{amscd}
\usepackage{amssymb}
\usepackage{latexsym}
\usepackage{stmaryrd}

\newtheorem{theorem}{Theorem}[section]
\newtheorem*{theorem*}{Theorem}

\newtheorem{lemma}[theorem]{Lemma}
\newtheorem*{lemma*}{Lemma}
\newtheorem{corollary}[theorem]{Corollary}
\newtheorem{proposition}[theorem]{Proposition}

\newtheorem{remark}[theorem]{Remark}
\newtheorem{definition}[theorem]{Definition}

\newcommand{\J}{\mathcal{J}}
\newcommand{\HH}{\mathcal{H}}

\newcommand{\s}{\mathcal{S}}
\newcommand{\A}{\mathcal{A}}
\newcommand{\B}{\mathcal{B}}

\newcommand{\V}{\mathcal{V}}
\newcommand{\W}{\mathcal{W}}
\newcommand{\T}{\mathcal{T}}

\newcommand{\C}{\mathbb{C}}
\newcommand{\AT}{\mathcal{A}_{tail}}
\newcommand{\p}{{\bf P}}
\usepackage[all]{xy}

\begin{document}
\title{A noncommutative  De Finetti theorem for boolean independence}

\author{Weihua Liu}
\maketitle
\begin{abstract}
We introduce a family of quantum semigroups and its natural coactions on  noncommutative polynomials. We define three invariance conditions for the joint distribution of sequences of selfadjoint noncommutative random variables associated with these coactions. For one of the invariance conditions, we show that the joint distribution of an infinite sequence of noncommutative random variables  satisfies it is equivalent to the fact that the sequence of the random variables is  identically distributed and boolean independent with respect to the conditional expectation onto  its tail algebra. This is a boolean analogue of  de Finetti theorem on exchangeable sequences. In the end of the paper, we also discuss the other two invariance conditions which lead to some trivial results.
\end{abstract}

\section{Introduction}
In classical probability, the study of random variables with distributional symmetries was started by the pioneering work of de Finetti on 2-point valued random variables. One of the most general versions of de Finetti's work states that an infinite sequence of random variables, whose joint distribution is invariant under all finite permutations, is conditionally independent and identically distributed.  One can see e.g. \cite{Ka} for an exposition on the classical de Finetti theorem for more details. Also, see \cite{HS}, Hewitt and Savage considered the distributional symmetries of random variables which are distributed on $X=E\times E\times E\times\cdots,$ where E is a compact Hausdorff space.  Later, in \cite{Sto}, an early noncommutative version of de Finetti theorem was  given by St\o rmer. His work focused on exchangeable states on the infinite reduced tensor product of $C^*$-algebras.
Roughly speaking, in noncommutative probability theory,  St\o rmer studied  symmetric states on  commuting noncommutative random variables. Recently, in \cite{Ko}, without the commuting relation, K\"ostler studied exchangeable sequences of noncommutative random variables  in $W^*-$ probability spaces with normal faithful states. In classical probability, if the second moment of a real valued random variable is 0, then the random variable is 0 a.e.. Faithfulness is a natural generalization of this property in noncommutative probability, readers are refered to  \cite{VDN}. K\"ostler  showed that exchangeable sequences of random variables possess some kind of factorization property, but the exchangeability does not imply any kind of universal relation. In other words, we can not expect to determine mixed moments of an exchangeable sequence of random variables in Speicher's universal sense \cite{Sp}. By strengthening  \lq\lq exchangeability" to invariance under a coaction of the free quantum permutations, in \cite{KS}, K\"ostler and  Speicher discovered that the de Finetti theorem has a natural analogue in Voiculescu's free probability theory(see \cite{VDN}).  Here, free quantum permutations refer to Wang's quantum groups $\A_s(n)$ in \cite{Wan}.

K\"ostler and  Speicher's work starts a systematic study of the probabilistic symmetries on noncommutative probability theory. Most of the further projects are developed by Banica, Curran and Speicher, see \cite{BCS},\cite{Cu},\cite{Cu2}. They showed their de Finetti type theorems in both of the classical(commutative) probability theory and the noncommutative probability theory under the invariance conditions of easy groups and easy quantum groups, respectively. All these works in noncommutative case are proceeded under the assumption that the state of the probability space is faithful. This is a natural assumption in free probability theory,  because in \cite{Dy}, Dykema showed that the free product of a family of $W^*$-probability spaces with normal faithful states is also a $W^*$-probability space with a normal faithful state. Thus the category of $W^*$-probability spaces with faithful states is closed under the free product construction. Since the family of $W^*$-probability spaces with normal faithful states is a  part of $W^*$-probability spaces with normal faithful states, one may ask what happens to probability spaces with states which are not necessarily faithful. More specific, what is the de Finetti type theorem  for more general noncommutative probability spaces?

Recall that in the noncommutative realm, besides the freeness and the classical independence, there are many other kinds of independence relations, e.g.  monotone independence \cite{Mu}, boolean independence\cite{SW}, type B independence \cite{BGN}  and more recently  two-face  freeness for pairs of random variables\cite{Vo}. All these types of independence are associated with certain products on probability spaces.  Among these products, in \cite{Sp}, Speicher showed that there are only two universal products on the unital noncommutative probability spaces, namely the tensor product and the free product. The corresponding independent relations associated with these two universal products are the classical independence and the free independence.  It was also showed in \cite{Sp} that there is a unique  universal product in the non-unital framework which is  called  boolean product. This non-unital universal product  provides a way to construct probability spaces with non-faithful states from probability spaces with faithful states.  The more general noncommutative probability spaces will be defined in section 6 which are called noncommutative probability spaces with non-degenerated states. We would expect that boolean independence plays the same role in noncommutative probability spaces with non-degenerated states as the classical independence and the freeness play in commutative probability spaces and noncommutative probability spaces with faithful states, respectively. The main purpose of this work is to give certain distributional symmetries  which can characterize conditionally boolean independence in  de Finetti theorem's form.

To proceed this work, we will construct a class of quantum semigroups $B_s(n)$'s and its sub quantum semigroups $\B_s(n)$'s. Then, we can define a coaction of $B_s(n)$ on the set of noncommutative polynomials with $n$-indeterminants. Unlike $B_s(n)$, there are two natural ways to define coactions of $\B_s(n)$ on the set of noncommutative polynomials. The first way considers the set of noncommutative polynomials as a linear space, the coaction of $\B_s(n)$  defined on the linear space will be called the linear coaction of $\B_s(n)$ on the set of noncommutative polynomials. The second way defines the  coaction of $\B_s(n)$  by considering the set of noncommutative polynomials as an algebra, the coaction of $\B_s(n)$ defined as a coaction on the algebra will be called the algebraic coaction of $\B_s(n)$ on the set of noncommutative polynomials. With these three  coactions of the quantum semigroups on the set of noncommutative polynomials with $n$-indeterminants, we can describe three invariance conditions for the joint distribution of any sequence of $n$ random variables $(x_1,..,x_n).$ We will show that the invariance conditions determined by the algebraic coaction of $\B_s(n)$ and the coaction of $B_s(n)$ are so strong such that if the joint distribution of the sequence of $n$ random variables $(x_1,...,x_n)$ satisfies one of the invariance conditions, then $x_1=x_2=\cdots=x_n$ or $x_1=x_2=\cdots=x_n=0$,respectively. In this paper, we are mainly concerned with the invariance conditions which are determined by the linear coactions of the quantum semigroups $\B_s(n)$'s.  Before proving the main theorems, we will study tail algebras in $W^*$ probability spaces with non-degenerated normal states.  There will be a brief discussion on why we should consider these more general spaces. Unlike $W^*$ probability spaces with faithful normal states, we will define two kinds of  tail algebras, one contains the unit of the original algebra  and the other may not.  As  K\"ostler did in  \cite{Ko}, we will define our conditional expectation by taking the WOT limit of \lq\lq shifts". One of the differences between our work and K\"ostler' result is that our tail algebra may not contain the unit of the original algebra.  Then, we will prove the following theorem for the two different cases (tail algebra with the unit of the original algebra or not):

\begin{theorem}
Let $(\A,\phi)$ be a $W^*$-probability space and $(x_i)_{i\in\mathbb{N}}$ be an infinite sequence of selfadjoint random variables which generate $\A$ as a von Neumann algebra and the unit of $\A$ is (not) contained in the WOT closure of the non unital algebra generated by $(x_i)_{i\in\mathbb{N}}$ . Then the following are equivalent:
\begin{itemize}
\item[a)] The joint distribution of $(x_i)_{i\in \mathbb{N}}$  satisfies the invariance condition associated with the linear coactions of the quantum semigroups $\B_s(n)$'s.

\item[b)]  The sequence $(x_i)_{i\in\mathbb{N}}$ is identically distributed and boolean independent with respect to the $\phi-$preserving conditional expectation $E$ onto the non unital(unital) tail algebra of the $(x_i)_{i\in\mathbb{N}}$
\end{itemize}
\end{theorem}

One can see the definitions of $A_s(n)$ and $\B_s(n)$ in section 2 and 3 for details. It should be mentioned here that Wang's quantum permutation group $A_s(n)$ is a quotient algebra of $\B_s(n)$ for each $n$. Moreover, both of the invariance conditions associated with the linear coactions and the algebraic coactions of the quantum semigroups $\B_s(n)$'s are stronger than the invariance condition associated with the quantum permutations $A_s(n)$'s

The paper is organized as follows: In Section 2, we recall the basic definitions and notation from the noncommutative probability theory, Wang's quantum groups and exchangeable sequence of random variables. In Section 3, we introduce our quantum semigroup $B_s(n)$ and its sub quantum semigroups $\B_s(n)$. Then, we introduce  a linear coaction of the quantum semigroup $\B_s(n)$ on the set of the noncommutative polynomials. We will define an invariance condition associated with the linear coaction of $\B_s(n)$.  In section 4, we have a brief discussion on the relation between freeness and boolean independence. We show that operator valued boolean independence implies operator valued freeness in some special cases. In section 5, we prove that the joint distribution of a finite sequence of $n$ boolean independent operator valued random variables  are invariant under the linear coaction of $\B_s(n)$. In section 6, we recall the properties of the tail algebra of any infinite exchangeable sequences of noncommutative variables and study the properties of the tail algebra under the boolean exchangeable condition. In section 7, we will prove the main theorems and provide some examples. In section 8, we define a coaction of $B_s(n)$ and an algebraic coaction of $\B_s(n)$ on the set of noncommutative polynomials in $n$ indeterminants. Then, we define the invariance conditions associated with these coactions. We will study the set of random variables $(x_1,...,x_n)$ whose joint distribution satisfies one of these  invariance conditions.

\section{Preliminaries and Notation}
\subsection{Noncommutative probability space}
We recall some necessary definitions and notation of noncommutative probability spaces. For further details, see texts \cite{KS}, \cite{NS}, \cite{BPV}, \cite{VDN}.
\begin{definition}\normalfont  A non-commutative probability space $(\A,\phi)$ consists of a unital algebra $\A$ and a linear functional $\phi:\A\rightarrow \C$. $(\A,\phi)$ is  called a $*$-probability space if $\A$ is a $*$-algebra and $\phi(xx^*)\geq 0$ for all $x\in\A$. $(\A,\phi)$ is called a  $W^*$-probability space if $\A$ is a $W^*$-algebra and $\phi$ is a normal state on it. We say $(\A,\phi)$ is tracial if
$$\phi(xy)=\phi(yx),\,\,\,\,\,\forall x,y\in\A.$$
The elements of $\A$ are called  random variables. Let $x\in\A$ be a random variable, then its distribution is a linear functional $\mu_x$ on $\C[X]$( the algebra of complex polynomials in one variable), defined by $\mu_x(P)=\phi(P(x))$.
\end{definition}

Note that we do not require the state on $W^*$-probability space to be tracial. We will specify the probability spaces we concern in section 6 and section 8. 

\begin{definition}\normalfont Let $I$ be an index set, the algebra of noncommutative polynomials in $|I|$ variables, $\C\langle X_i|i\in I\rangle$, is the  linear span of $1$ and noncommutative monomials of the form $X^{k_1}_{i_1}X^{k_2}_{i_2}\cdots X^{k_n}_{i_n}$ with $i_1\neq i_2\neq\cdots \neq i_n\in I$ and all $k_j$'s are positive integers. For convenience we will use $\C\langle X_i|i\in I \rangle_0$ to denote the set of noncommutative polynomials without constant term.

Let $(x_i)_{i\in I}$ be a family of random variables in a noncommutative probability space $(\A,\phi)$. Their  joint distribution is a linear functional $\mu:\C\langle X_i|i\in I \rangle\rightarrow \C$ defined by
$$\mu(X^{k_1}_{i_1}X^{k_2}_{i_2}\cdots X^{k_n}_{i_n})=\phi(x^{k_1}_{i_1}x^{k_2}_{i_2}\cdots x^{k_n}_{i_n}),$$
and $\mu(1)=1$.
\end{definition}

\begin{remark}\normalfont
In general, the joint distribution depends on the order of the random variables, e.g. let $I=\{1,2\}$, then $\mu_{x_1,x_2}$ may not equal $\mu_{x_2,x_1}$. According to our notation, $\mu_{x_1,x_2}(X_1X_2)=\phi(x_1x_2)$, but $\mu_{x_2,x_1}(X_1X_2)=\phi(x_2x_1)$.
\end{remark}

\begin{definition}\normalfont Let $(\A,\phi)$ be a noncommutative probability space, a family of unital subalgebras $(\A_i)_{i\in I}$ is said to be free if
$$\phi(a_1\cdots a_n)=0,$$
whenever $a_k\in\A_{i_k}$, $i_1\neq i_2\neq\cdots\neq i_n$ and $\phi(a_k)=0$ for all $k$. Let $(x_i)_{i\in I}$ be a family of random variables and $\A_i$'s be  the unital subalgebras generated by $x_i$'s, respectively. We say the family of  random variables $(x_i)_{i\in I}$ is free if the family of unital subalgebras $(\A_i)_{i\in I}$ is free.
\end{definition}

\begin{definition}\normalfont Let  $(\A,\phi)$ be a noncommutative probability space, a family of (usually non-unital) subalgebras $\{\A_i|i\in I \}$ of $\A$ is said to be boolean independent if
$$\phi(x_1x_2\cdots x_n)= \phi(x_1)\phi(x_2)\cdots \phi(x_n)$$
whenever $x_k\in\A_{i(k)}$ with $i(1)\neq i(2)\neq\cdots \neq i(n)$. A set of random variables $\{x_i\in\A|i\in I\}$ is said to be boolean independent if the family of non-unital subalgebras $\A_i$, which are generated by $x_i$ respectively, is boolean independent.
\end{definition}
One refers to \cite{Fr} for more details of boolean product of random variables. Since the framework for boolean independence is a non-unital algebra in general, we will not require  our operator valued probability spaces to be unital: 
\begin{definition} \normalfont An operator valued probability space $(\A,\B, E:\A\rightarrow \B)$ consists of an algebra  $\A$, a subalgebra $\B$ of $\A$ and a $\B-\B$ bimodule linear map $E:\A\rightarrow \B$ i.e.
 $$E[b_1ab_2]=b_1E[a]b_2,$$
for all $b_1,b_2\in\B$ and $a\in\A$. According to the definition in \cite{St}, we call $E$ a conditional expectation from $\A$ to $\B$ if $E$ is onto, i.e. $E[\A]=\B$. The elements of $\A$ are called random variables.
\end{definition}
In operator valued free probability theory, $\A$ and $\B$ are unital and  have the same unit 

\begin{definition}\normalfont Given an algebra $\B$, we denote by $\B\langle X\rangle$  the algebra which is freely generated by $\B$ and the indeterminant $X$. Let $1_X$ be the identity of $\C\langle X\rangle$, then $\B\langle X\rangle$ is set of linear combinations of the elements in $\B$ and the noncommutative monomials $b_0Xb_1Xb_2\cdots b_{n-1}Xb_n$ where $b_k\in\B\cup \{\C 1_X\}$ and $n\geq 1$. The elements in  $\B\langle X\rangle$ are called $\B$-polynomials. In addition,  $\B\langle X\rangle_0$ denotes the subalgebra of $\B\langle X\rangle$ which doesn't contain the constant term i.e. the linear span of the noncommutative monomials $b_0Xb_1Xb_2\cdots b_{n-1}Xb_n$ where $b_k\in\B\cup \{\C 1_X\}$ and $n\geq 1$. $\B\langle X\rangle_0$.
\end{definition}

Given an operator valued probability space $(\A,\B,E:\A\rightarrow \B)$ such that $\A$ and $\B $ are unital. A family of unital subalgebras $\{\A_i\supset \B \}_{i\in I}$  is said to be freely independent with respect to $E$ if 
$$E[a_1\cdots a_n]=0,$$
whenever $i_1\neq i_2\neq \cdots\neq i_n$, $a_k\in \A_{i_k}$ and $E[a_k]=0$ for all $k$. A family of $ (x_i)_{i\in I}$ is said to be free independent over $\B$, if the unital subalgebras $\{\A_i\}_{i\in I}$ which are generated by $x_i$ and $B$ respectively is free, or equivalently 
$$E[p_1(x_{i_1})p_2(x_{i_2})\cdots p_n(x_{i_n})]=0,$$
whenever $i_1\neq i_2\neq \cdots\neq i_n$, $p_1,...,p_n\in \B\langle X\rangle$ and $E[p_k(x_{i_k})]=0$ for all $k$.\\
Let  $\{x_i\}_{i\in I}$  be a family of random variables in an operator valued probability space $(\A,\B,E:\A\rightarrow \B)$. $\A$, $\B$ are not necessarily unital.  $\{x_i\}_{i\in I}$ is  said to be boolean independent over $\B$ if for all $i_1,...,i_n\in I$, with $i_1\neq i_2\neq\cdots\neq i_n$ and all $\B$-valued polynomials $p_1,...,p_n\in\B\langle X\rangle_0$ such that
$$E[p_1(x_{i_1})p_2(x_{i_2})\cdots p_n(x_{i_n})]=E[p_1(x_{i_1})]E[p_2(x_{i_2})]\cdots E[p_n(a_{i_n})].$$

\subsection{Wang's quantum permutation groups }

In \cite{Wan}, Wang introduced the following quantum groups $A_s(n)$'s.
\begin{definition} \normalfont $A_s(n)$ is defined as the universal unital $C^*$-algebra generated by elements $u_{ij}$ $(i,j=1,\cdots n)$ such that we have
\begin{itemize}
\item each $u_{ij}$ is an orthogonal projection, i.e. $u_{ij}^*=u_{ij}=u_{ij}^2$ for all $i,j=1,...,n$.
\item the elements in each row and column of $u=(u_{ij})_{i,j=1,...,n}$ form a partition of unit, i.e. are orthogonal and sum up to 1: for each $i=1,\cdots,n$ and $k\neq l$ we have
$$
\begin{array}{rcl}
u_{ik}u_{il}=0&\text{and}& u_{ki}u_{li}=0;\\
\end{array}$$
and for each $i=1,\cdots, n$ we have
$$\sum\limits_{k=1}^nu_{ik}=1=\sum\limits_{k=1}^nu_{ki}.$$
\end{itemize}
\end{definition}

$A_s(n)$ is a compact quantum group in the sense of Woronowicz \cite{Wo2}, with comultiplication, counit and antipode given by the formulas:
$$\Delta u_{ij}=\sum\limits_{k=1}^n u_{ik}\otimes u_{kj}$$
$$\epsilon (u_{ij})=\delta_{ij} $$
$$S(u_{ij})=u_{ji}.$$
The right coaction of $A_s(n)$ on $\C\langle X_1,...,X_n\rangle$ is a linear map $\alpha:\C\langle X_1,...,X_n\rangle\rightarrow \C\langle X_1,...,X_n\rangle\otimes A_s(n)$ given by:
$$\alpha (X_{i_1}X_{i_2}\cdots X_{i_m})=\sum\limits_{j_1,...,j_m=1}^n X_{j_1}X_{j_2}\cdots X_{j_m}\otimes u_{j_1,i_1}u_{j_2,i_2}\cdots u_{j_m,i_m},$$
where $\otimes$ denotes the algebraic tensor product.

In the earlier papers, $\alpha$ is defined as an algebraic homomorphism. We emphasis on the 
linearity here because we will define some coactions of our quantum semigroups on  noncommutative polynomials in a similar way.  The right coaction has the following property:
$$(\alpha\otimes id)\alpha=(id\otimes \Delta)\alpha.$$

Let $(x_i)_{i\in\mathbb{N}}$ be an infinite sequence of random variables in a noncommutative probability space $(\A,\phi)$, the sequence is said to be quantum exchangeable if their joint distribution is invariant under Wang's quantum permutation groups, i.e. for all $n$, we have 
$$\mu_{x_1,...x_n}(p)1_{A_s(n)}=\mu_{x_1,...,x_n}\otimes id_{A_s(n)}(\alpha (p)),$$  
where $\mu_{x_1,...,x_n}$ is the joint distribution of $x_1,...,x_n$ with respect to $\phi$ and $p\in \C\langle X_1,...,X_n\rangle $. For example, if $p=X_{i_1}X_{i_2}\cdots X_{i_m}$, then the equation above can be written as:
$$
\begin{array}{rcl}
\phi(x_{i_1}x_{i_2}\cdots x_{i_m})1_{A_s(n)}&=&\mu_{x_1,...x_n}((X_{i_1}X_{i_2}\cdots X_{i_m})1_{A_s(n)}\\
&=&\mu_{x_1,...,x_n}\otimes id_{A_s(n)}(\sum\limits_{j_1,...,j_m=1}^n X_{j_1}X_{j_2}\cdots X_{j_m}\otimes u_{j_1,i_1}u_{j_2,i_2}\cdots u_{j_m,i_m})\\
&=&\sum\limits_{j_1,...,j_m=1}^n \phi(x_{j_1}x_{j_2}\cdots x_{j_m})u_{j_1,i_1}u_{j_2,i_2}\cdots u_{j_m,i_m},
\end{array}
$$
whenever $i_1\neq i_2\neq \cdots\neq i_n$.\\
 Let $S_n$ be the permutation group on $\{1,...,n\}$. The joint distribution of $(x_i)_{i\in\mathbb{N}}$ is said be exchangeable if for all $n$, $\sigma\in S_n$, we have 
$$\mu_{x_1,...x_n}=\mu_{x_{\sigma(1)},...,x_{\sigma(n)}},$$ 
where $\mu_{x_1,...,x_n}$ is the joint distribution of $x_1,...,x_n$ with respect to $\phi$ . It is showed in \cite{KS} that quantum exchangeability implies classical exchangeability.

\section{Quantum semigroups}
Our probabilistic symmetries will be given by the invariance conditions associated with certain coactions of our quantum semigroups.  First, we recall the related definitions  and notation of quantum semigroups.

A quantum space is an object of the category dual to the category of $C^*$-algebras(\cite{Wo}). For any $C^*$-algebras $A$ and $B$, the set of  morphisms Mor$(A,B)$ consists of all $C^*$-algebra homomorphisms acting from $A$ to $M(B)$, where $M(B)$ is the multiplier algebra of $B$, such that $\phi(A)B$ is dense in $B$. If $A$ and $B$ are unital $C^*$-algebras, then all unital $C^*$- homomorphisms from $A$ to $B$ are in Mor(A,B). In \cite{So}, 

\begin{definition}
By a quantum semigroup we mean a $C^*$-algebra $\A$ endowed with an additional structure  described by a morphism $\Delta\in Mor(\A,\A\otimes\A)$ such that
     $$(\Delta\otimes id_{\A})\Delta=(id_{\A}\otimes\Delta)\Delta.$$

\end{definition}
In other words, $\Delta$ defines a comultiplication on $\A$. Here the tensor product $\otimes $ denotes the minimal tensor product $\otimes_{min}$.\\
Now, we turn to introduce our quantum semigroups:

\noindent{\bf Quantum semigroups ($B_s(n)$, $\Delta$):} The algebra $B_s(n)$ is defined as the universal unital $C^*$-algebra generated by elements $u_{i,j}$ $(i,j=1,\cdots n)$ and a projection ${\bf P}$ such that we have
\begin{itemize}
\item each $u_{i,j}$ is an orthogonal projection, i.e. $u_{i,j}^*=u_{i,j}=u_{i,j}^2$ for all $i,j=1,\cdots, n$,
\item $$
\begin{array}{rcl}
u_{i,k}u_{i,l}=0&\text{and}& u_{k,i}u_{l,i}=0,\\
\end{array}
$$
whenever $k\neq l$
\item For all $1\leq i\leq n$,  ${\bf P}=\sum\limits_{k=1}^{n}u_{k,i}{\bf P}$.
\end{itemize}
We will denote the unite of $B_s(n)$ by $I$, the projection ${\bf P}$ is called the invariant projection of $B_s(n)$. \\
On this unital $C^*$-algebra, we can define a unital $C^*$-homomorphism $$\Delta:B_s(n)\rightarrow B_s(n)\otimes B_s(n)$$ by the following formulas:

$$\Delta u_{i,j}=\sum\limits_{k=1}^nu_{i,k}\otimes u_{k,j}$$
and $$\Delta{\bf P}={\bf P}\otimes {\bf P},\,\,\,\, \Delta I=I\otimes I.$$

We will see that $(B_s(n),\Delta)$ is a quantum semigroup. To show this we need to check that $\Delta$ defines a unital $C^*$-homomorphism from $B_s(n)$ to $B_s(n)\otimes B_s(n)$ and satisfies the comultiplication condition :\\
First, $\Delta u_{i,j}=\sum\limits_{k=1}^nu_{i,k}\otimes u_{k,j}$ is a projection because  $u_{i,k}$, $u_{k,j}$ are projections and $u_{i,k}u_{i,l}=0$ if $k\neq l$,  $u_{i,k}\otimes u_{k,j}$'s are orthogonal to each other. Also, $\Delta{\bf P}={\bf P}\otimes {\bf P}$ is a projection. Let $l\neq m$, then
$$
\begin{array}{rcl}
\Delta(u_{i,l})\Delta{u_{i,m}}&=&(\sum\limits_{k=1}^n u_{i,k}\otimes u_{k,l})(\sum\limits_{j=1}^nu_{i,j}\otimes u_{j,m})\\
&=&\sum\limits_{k,j=1}^nu_{i,k}u_{i,j}\otimes u_{k,l}u_{j,m}\\
&=&\sum\limits_{k=1}^nu_{i,k}\otimes u_{k,l}u_{k,m}\\
&=&0.\\
\end{array}
$$
The same, we have $ \Delta(u_{l,i})\Delta{u_{m,i}}=0$, for $m\neq l$. Moreover, we have  
$$
\begin{array}{rcl}
\Delta(\sum\limits_{l=1}^nu_{l,i})\Delta{\bf P}&=&(\sum\limits_{l,k=1}^nu_{l,k}\otimes u_{k,i}){\bf P}\otimes {\bf P}\\
&=&\sum\limits_{l,k=1}^nu_{l,k}{\bf P}\otimes u_{k,i}{\bf P}\\
&=&\sum\limits_{k=1}^n{\bf P}\otimes u_{k,i}{\bf P}\\
&=&{\bf P}\otimes {\bf P}.\\
\end{array}
$$
and
$\Delta$ sends the unit of $B_s(n)$ to the unit of $B_s(n)\otimes B_s(n)$.
Therefore, $\Delta$ defines a unital $C^*$-homomorphism on $B_s(n)$ by the universality of $B_s(n)$.\\

The comultiplication condition holds, because on the generators we have:
$$(\Delta\otimes id_{\A})\Delta u_{i,j}=\sum\limits_{k,l=1}^nu_{ik}\otimes u_{k,l}\otimes u_{l,j}=(id_{\A}\otimes\Delta)\Delta u_{i,j}$$
$$(\Delta\otimes id_{\A})\Delta{\bf P}={ \bf P}\otimes{\bf P}\otimes{\bf P}=(id_{\A}\otimes\Delta)\Delta{\bf P}$$
$$(\Delta\otimes id_{\A})\Delta I=I\otimes I\otimes I=(id_{\A}\otimes\Delta)\Delta I. $$
Therefore, $(B_s(n),\Delta)$ is a quantum semigroup.

\begin{remark}
If we let the invariant projection to  be  the identity, then we get Wang's free quantum permutation group. Therefore, $A_s(n)$ is a quotient $C^*$-algebra of $B_s(n)$, i.e. there exists a unital $C^*$-homomorphism $\beta:B_s(n)\rightarrow A_s(n)$ such that $\beta$ is surjective.
\end{remark}
\noindent Now, we provide some nontrivial representations of $B_s(n)$'s:\\
Let $\mathbb{C}^6$ be the standard  6-dimensional complex Hilbert space with orthonormal basis $v_1,...,v_6$. Let
$$
\begin{array}{ccc}
P_{11}=P_{v_1+v_2},\,\,\,&P_{21}=P_{v_3+v_4},\,\,\,&P_{13}=P_{v_5+v_6},\\
P_{21}=P_{v_3+v_6},\,\,\,&P_{22}=P_{v_5+v_2},\,\,\,&P_{23}=P_{v_1+v_4},\\
P_{31}=P_{v_4+v_5},\,\,\,&P_{32}=P_{v_1+v_6},\,\,\,&P_{33}=P_{v_2+v_3}.
\end{array}
$$
and ${P}=P_{v_1+v_2+v_3+v_4+v_5+v_6}$, where $P_v$ denotes the  one dimensional orthogonal projection onto the subspace spaned by $v$. Then the unital algebra generated by $P_{i,j}$ and $P$ gives a representation $\pi$ of $B_s(3)$ on $\C^6$ by the following formulas on the generators of $\B_s(3)$:
$$\pi(I)=I_{\C^6},\,\,\,\pi(u_{i,j})=P_{ij},\,\,\,\pi(\p)=P.$$
$\pi$ is well defined by the universality of $B_s(3).$\\
Moreover, the matrix form for $P_{1,1}$ and $P$ with respect to the basis are $$
P_{11}=1/2\left(
\begin{array}{llllll}
1&1&0&0&0&0\\
1&1&0&0&0&0\\
0&0&0&0&0&0\\
0&0&0&0&0&0\\
0&0&0&0&0&0\\
0&0&0&0&0&0
\end{array}
\right)
 \text{and }
P=1/6\left(
\begin{array}{llllll}
1&1&1&1&1&1\\
1&1&1&1&1&1\\
1&1&1&1&1&1\\
1&1&1&1&1&1\\
1&1&1&1&1&1\\
1&1&1&1&1&1
\end{array}
\right),$$
then we have
$$PP_{11}P=1/18\left(
\begin{array}{llllll}
1&1&1&1&1&1\\
1&1&1&1&1&1\\
1&1&1&1&1&1\\
1&1&1&1&1&1\\
1&1&1&1&1&1\\
1&1&1&1&1&1
\end{array}
\right)=1/3P.$$
In general,we have
\begin{lemma}\label{4.1}
Let $v_1,..v_{2n}$ be an orthonormal basis of the standard $2n-$dimensional Hilbert space $\mathbb{C}^{2n}$, and let $v_k=v_{k+2n}$ for all $k\in\mathbb{Z}$, let
$$P_{i,j}=P_{v_{2(i-j)+1}+v_{2(j-i)+2}},$$
where $P_v$ is the orthogonal projection the one dimensional subspace generated by the  vector $v$ and $P=P_{v_1+v_2+\cdots+v_{2n}}$, ${\bf 1}$ is the identity of $B(\C^{2n})$ . Then $\{P_{i,j}\}_{i,j=1,...,n}$, $P$ and ${\bf 1}$ satisfy the  defining conditions of the algebra $B_s(n)$,
\end{lemma}
\begin{proof}
It is easy to see that  the inner product
$$\langle v_{2(i-j)+1}+v_{2(j-i)+2},v_{2(i-k)+1}+v_{2(k-i)+2}\rangle=2\delta_{j,k},$$
so $P_{ik}P_{ij}=0$ if $j\neq k$. The same $P_{ki}P_{ji}=0$ if $k\neq j$. Fix $i$, we see that $v_1+v_2+\cdots+v_{2n}\in span\{v_{2(i-j)+1}+v_{2(j-i)+2}|j=1,...n\}$, so $\sum\limits_{k=1}^n P_{ik}P=P.$
\end{proof}
Therefore, by lemma \ref{4.1}, there exists a representation $\pi$ of $B_s(n)$ on $\C^{2n}$ which is defined by the following formulas:
$$\pi(1_{B_s(n)})={\bf 1},\,\,\,\, \pi(\p)=P$$
and $$\pi(u_{i,j})=P_{i,j},$$
for all $i,j=1,...,n$.

\noindent Now, we turn to introduce a sub  quantum semigroup of $(B_s(n), \Delta)$. Since $\p\neq I$ is a projection in $B_s(n)$, $\B_s(n)=\p B_s(n)\p$ is a $C^*$-algebra with identity $\p$ and generators $$ \{\p u_{i_1,j_1}\cdots u_{i_k,j_k}\p|i_1,j_1,...i_k,j_k\in\{1,...n\},k\geq0\}.$$
 If we restrict the  comultiplication $\Delta$ onto $\B_s(n)$, then we have
$$\Delta(\p u_{i_1,j_1}\cdots u_{i_k,j_k}\p)=(\p\otimes\p)(\sum\limits_{l_1,...l_k=1}^nu_{i_1,l_1}\cdots u_{i_k,l_k}\otimes u_{l_1,j_1}\cdots u_{l_k,j_k})(\p\otimes\p),$$
which is contained in $\B_s(n)\otimes\B_s(n)$. Therefore, $(\B_s(n),\Delta)$ is also a  quantum semigroup and $\p$ is the identity of $\B_s(n)$. We will call $\B_s(n)$ the boolean permutation quantum semigroup of $n$. 
\begin{remark}
If we require $\p u_{i,j}=u_{i,j}\p$ for all  $i,j=1,...,n$, then the universal algebra  we constructed in the above way is exactly Wang's quantum permutation group. Therefore, $A_s(n)$ is also a quotient algebra of $\B_s(n)$.
\end{remark}

In the following definition, $\otimes$ denotes the tensor product for  linear spaces:
\begin{definition}
Let $\s=(\A,\Delta)$ be a quantum semigroup and $\V$ be a complex vector space, by a (right) coaction of the quantum group $\s$ on $\V$ we mean a linear map $\L:\V\rightarrow\V\otimes \A$ such that
$$(\L\otimes id)\L=(id\otimes \Delta)\L.$$
We say a linear functional $\omega: \V\rightarrow\C $ is invariant under $\L$ if
$$(\omega\otimes id)\L(v)=\omega(v)I_{\A}, $$
where $I_{\A}$ is the identity of $\A$.\\
Given a complex vector space $\W$, We say a linear map $T:\V\rightarrow\W$ is invariant under $\L$ if
$$(T\otimes id)\L(v)=T(v)\otimes I_{\A}. $$
\end{definition}

\begin{remark}
This definition is about  coactions on linear spaces but not  coactions on algebras.
\end{remark}

Let $\C\langle X_1,...,X_n\rangle$ be the set of noncommutative polynomials in $n$ indeterminants, which is a linear space over $\C$ with basis $X_{i_1}\cdots X_{i_k}$ for all integer $k\geq0$ and $i_1,...,i_k\in\{1,...n\}$.\\
Now, we define a right coaction $\L_n$ of $\B_s(n)$ on $\C\langle X_1,...,X_n\rangle$ as follows:
$$\L_n(X_{i_1}\cdots X_{i_k})=\sum\limits_{j_1,...j_k=1}^nX_{j_1}\cdots X_{j_k}\otimes \p u_{j_1,i_1}\cdots u_{j_n,i_n}\p.$$
It is a well defined coaction of $\B_s(n)$ on $\C\langle X_1,...,X_n\rangle$, because:

$$
\begin{array}{rl}
&(\L_n\otimes id)\L_n(X_{i_1}\cdots X_{i_k})\\
=&(\L_n\otimes id)\sum\limits_{j_1,...j_k=1}^nX_{j_1}\cdots X_{j_k}\otimes \p u_{j_1,i_1}\cdots u_{j_n,i_n}\p\\

=&\sum\limits_{j_1,...j_k=1}^n\sum\limits_{l_1,...l_k=1}^nX_{l_1}\cdots X_{l_k}\otimes\p u_{l_1,j_1}\cdots u_{l_n,j_n}\p \otimes \p u_{j_1,i_1}\cdots u_{j_n,i_n}\p\\

=&\sum\limits_{l_1,...l_k=1}^n X_{l_1}\cdots X_{l_k}\otimes(\sum\limits_{j_1,...j_k=1}^n\p u_{l_1,j_1}\cdots u_{l_n,j_n}\p \otimes \p u_{j_1,i_1}\cdots u_{j_n,i_n}\p)\\

=&\sum\limits_{l_1,...l_k=1}^n X_{l_1}\cdots X_{l_k}\otimes(\Delta\p u_{l_1,i_1}\cdots u_{l_n,i_n}\p)\\

=&(id\otimes \Delta)\sum\limits_{j_1,...j_k=1}^nX_{l_1}\cdots X_{l_k}\otimes(\p u_{l_1,i_1}\cdots u_{l_n,i_n}\p)\\
=&(id\otimes \Delta)\L_n(X_{i_1}\cdots X_{i_k}).\\
\end{array}
$$

\noindent We will call $\L_n$ the linear coaction of $\B_s(n)$ on $\C\langle X_1,...,X_n\rangle.$ The algebraic coaction will  be defined in section 7.

\begin{lemma} 
Let $\L_n$ be the linear coaction of $\B_s(n)$ on $\C\langle X_1,...,X_n\rangle$, $\{u_{i,j}\}_{i,j=1,...n}$ and $\p$ be the standard generators of $\B_s(n)$. Then,
$$\L_n(p_1(X_{i_1})\cdots p_k(X_{i_k}))=\sum\limits_{j_1,...j_k=1}^n p_1(X_{j_1})\cdots  p_k(X_{j_k})\otimes \p u_{j_1,i_1}\cdots u_{j_k,i_k}\p,$$
for all $i_1\neq i_2\neq \cdots \neq i_k$ and $p_1,...p_k\in \C\langle X\rangle.$
\end{lemma}
\begin{proof}
Since the map is linear, it suffices to show that the equation holds by assuming  $p_l(X)=X^{t_l}$ where $t_l\geq 1 $ for all $l=1,...k$. Then, we have  
$$
\begin{array}{rcl}
& &\L_n(\underbrace{x_{i_1}\cdots x_{i_1}}_{t_1\,\text{ times}}\cdots \underbrace{x_{i_1}\cdots x_{i_k}}_{t_k\,\text{ times}})  \\
&=& \sum\limits_{j_{1,1},...j_{1,t_1},...,j_{k,1},...j_{k,t_k}=1}^n x_{j_{1,1}}\cdots x_{j_{1,t_1}}\cdots x_{j_{k,1}}\cdots x_{j_{k,t_k}}\otimes \p u_{j_{1,1}i_1}\cdots u_{j_{1,t_1}i_1}\cdots\p.
\end{array}
$$ 
Notice that $ u_{j_{m,s}i_m}u_{j_{m,s+1} i_m}=\delta_{j_{m,s},j_{m,s+1}}u_{j_{m,s}i_m}$, the right hand side of the above equation becomes
$$\sum\limits_{j_1,...,j_k=1}^n x_{j_1}^{t_1}\cdots x_{j_k}^{t_k}\otimes \p u_{j_1 i_1}\cdots u_{j_k,j_k}\p.$$
The proof is now completed
\end{proof}

We will be using  the following invariance condition to characterize conditionally boolean independence.

\begin{definition}
Let  $(\A,\phi)$ be a noncommutative probability space and $(x_i)_{i
\in \mathbb{N}}$ be an infinite sequence of random variables in $\A$, we say the joint distribution satisfies the invariance conditions associated with the linear coactions of the boolean quantum permutation semigroups $\B_s(n)$ if for all $n$, we have 
$$\mu_{x_1,...,x_n}(p)\p=\mu_{x_1,...,x_n}\otimes id_{\B_s(n)}(\L_n p ) $$
for all $p\in \C\langle X_1,...,X_n\rangle$, where $\mu_{x_1,...x_n}$ is the joint distribution of $x_1,...,x_n$.
\end{definition}

Let $\{\bar u_{ij}\}_{i,j=1,..,n}$ be the standard generators of $\A_s(n)$, and $\{ u_{ij}\}_{i,j=1,..,n}\cup\{\p\}$ be the standard generators of $B_s(n)$, then there exists  a $C^*$-homomorphism $\beta:B_s(n)\rightarrow\A_s(n)$ such that:
$$\beta(u_{ij})=\bar u_{ij},\,\,\,\,\,\beta(\p)=1_{\A_s(n)}.$$ 
The $C^*$-homomorphism is well defined because of the universality of $B_s(n)$. Let $p=X_{i_1}\cdots X_{i_k}\in \C\langle X_1,...,X_n\rangle$, then 
$$\mu_{x_1,...,x_n}(p)\p=\mu_{x_1,...,x_n}\otimes id_{\B_s(n)}(\L_n p )$$ implies
$$\mu_{x_1,...,x_n}(p)\p=\mu_{x_1,...,x_n}\otimes id_{\B_s(n)}(\L_n p ) $$
$$\mu_{x_1,...,x_n}(X_{i_1}\cdots X_{i_k})\p=\sum\limits_{j_1,...j_k=1}^n(\mu_{x_1,...,x_n}\otimes id_{\B_s(n)})(X_{j_1}\cdots X_{j_k}\otimes \p u_{j_1,i_1}\cdots u_{j_n,i_n}\p).$$
Now, apply $\beta$ on both sides of the above equation, we get
$$\mu_{x_1,...,x_n}(X_{i_1}\cdots X_{i_k})1_{\A_s(n)}=\sum\limits_{j_1,...j_k=1}^n(\mu_{x_1,...,x_n}\otimes id_{\A_s(n)})(X_{j_1}\cdots X_{j_k}\otimes \bar u_{j_1,i_1}\cdots \bar u_{j_n,i_n}),$$
which is the free quantum invariance condition. Since $p$ is arbitrary, we have the following:
\begin{proposition}
Let  $(\A,\phi)$ be a noncommutative probability space and $(x_i)_{i=1,...,n}$ be a sequence of random variables in $\A$, the joint distribution of $(x_i)_{i=1,...,n}$ is invariant under the free quantum permutations $\A_s(n)$ if it satisfies the invariance condition associated with the linear coaction of the boolean quantum permutation semigroup $\B_s(n)$.

\end{proposition}

\section{Boolean independence and freeness}
In this section, we will show that operator valued boolean independent variables are  sometimes operator valued free independent. Therefore, we should not be surprised that the joint distribution of any sequence of identically boolean independent random variables is invariant under the coaction of the free quantum permutations. Especially, in section 7, operator valued boolean independent variables are  always operator valued free independent when we construct our conditional expectation in the  unital-tail algebra case. The properties are related to the $C^*-$ algebra unitalization. We provide a brief review here:\\
To every $C^*$ algebra $\A$ one can associate a unital $C^*$ algebra $\bar\A$ which contains $\A$ as a two-sided ideal and with the property that the quotient $C^*$-algebra $\bar\A/\A$ is isomorphic to $\C$. Actually, $\bar\A=\{x\bar I+a|x\in\C,a\in\A\}$, where $\bar I$ is the unit of $\bar\A$. We will denote $x\bar I+a$ by $(x,a)$ where $x\in\C$ and $a\in \A$, then we have 
$$(x,a)+(y,b)=(x+y,a+b),\,\,\,\,(x,a)(y,b)=(xy,ab+a+b),\,\,\,\,(x,a)^*=(\bar x,a^*).$$

 Let $(\A,\B,E)$ be an operator-valued probability space where $\A$ and $\B$ are not necessarily unital.  Let $\bar{\A}$ and $\bar{\B}$ be the unitalization defined above, then we can extend $\rho$ to $\bar{\rho}$ s.t $(\bar{\A},\bar{\B},\bar{E})$ is also an operator-valued probability space where $\bar{E}$ is a conditional expectation on $\bar{\A}$.\\
It is natural to define $\bar{E}$ as 
$$\bar{E}[x,a]=(x,E[a]).$$
$\bar{E}[(1,0)]=(1,0)$, so $\bar{E}$ is unital. The linear property is easy to check.\\
Take $(x_1,b_1), (x_2,b_2)\in\bar{\B}$ and $(y,a)\in\bar{\A}$, we have 
$$
\begin{array}{rcl}
\bar{E}[(x_1,b_1)(y,a)(x_2,b_2)]&=&\bar{E}[x_1yx_2,x_1x_2a+yx_2b+x_2b_1a+x_1b_2+yb_1b_2+b_1ab_2]\\
&=&(x_1yx_2,E[x_1x_2a+yx_2b+x_2b_1a+x_1b_2+yb_1b_2+b_1ab_2)]\\
&=&(x_1yx_2,x_1x_2E[a]+yx_2b+x_2b_1E[a]+x_1b_2+yb_1b_2+b_1E[a]b_2)\\
&=&(x_1,b_1)(y,E[a])(x_2,b_2)\\
&=&(x_1,b_1)\bar{E}[(y,a)](x_2,b_2).
\end{array}
$$
It is obvious that $\bar E^2=\bar E$. Hence, $\bar{E}$ is a $\bar{\B}$-$\bar\B$ bimodule from the unital algebra $\bar\A$ to the unital subalgebra $\bar\B$,  i.e. a conditional expectation.

\begin{proposition}
Let  $(\A,\B, E):\A\rightarrow \B$ be an operator valued probability space, $\{\A_i\}_{i\in I}$ be a $\B$-boolean independent family of sub-algebras and $\B\subset \A_i$ for all $i$. Then, in  the unitalization operator probability space $(\bar\A,\bar\B, \bar E)$,
 $\{\bar {\A}_i\}_{i\in I}$ is a $\bar\B$-free independent family of sub-algebras.
\end{proposition}
\begin{proof}
Let $(x,a)\in \bar\A$, where $a\in \A$ and $x$ is a complex number, then $\bar E[(x,a)]=(x,E[a])$, thus $\bar E[(x,a)]=0$ iff $x=0$ and $E[a]=0$. \\
Now, we can check the freeness directly. Let $(x_k,a_k)\in\bar\A_{i_k}$, i.e $a_k\in \A_{i_k}$ and $x_i$'s are complex numbers, for $k=1,\cdots,n$ and $\bar E[x_k,a_k]=0$ and $i_1\neq i_2\neq\cdots\neq i_n$, then we have  $x_k=0$ for all $k=1,\cdots,n$ and 

$$
\begin{array}{rcl}
\bar{E}[(x_1,a_1)(x_2,a_2)\cdots(x_n,a_n)]&=&\bar{E}[(0,a_1)(0,a_2)\cdots(0,a_n)]\\
&=&\bar E[(0,a_1a_2\cdot a_n)]\\
&=&(0,E[a_1a_2\cdots a_n)]\\
&=&(0,E[a_1]E[a_2]\cdots E[a_n])\\
&=&(0,0)=0.
\end{array}
$$
and $\bar \B\subset\bar\A_i$ for all $i$.
\end{proof}

The examples for this proposition will be given in section 7.3. By checking the conditions for operator valued freeness directly as we did in the above theorem, we have  
\begin{corollary}
Let  $(\A,\B, E):\A\rightarrow \B$ be an operator valued probability space, $\{\B\subset\A_i\}_{i\in I}$ be a $\B$-free independent family of sub-algebras. Then, in
 their unitalization operator probability space $(\bar\A,\bar\B, \bar E)$,
 $\{\bar {\A}_i\}_{i\in I}$ is a $\bar\B$-free independent family of sub-algebras.
\end{corollary}

\section{operator valued Boolean random variables are invariant under Boolean quantum permutations}

Let $\B_s(n)$ be the boolean permutation quantum  semigroup of $n$ with standard generators $\{u_{i,j}\}_{i,j=1,\cdots,n}$ and $\p$.  In this section, we prove that the joint distribution of $n$ boolean independent operator valued random variables are invariant under the linear coactions of $\B_s(n)$. The following equality is the key to the proof of the statement:

\noindent Fix $k$ and $1\leq i_1,\cdots ,i_k\leq n$, we have
$$
\begin{array}{rcl}
&&\sum\limits_{j_1,\cdots,j_k=1}^{n}\p u_{i_1,j_1}\cdots u_{i_k,j_k}\p\\
&=&\sum\limits_{j_1,\cdots,j_{k-1}=1}^{n}\p u_{i_1,j_1}\cdots u_{i_{k-1},j_k-1}(\sum\limits_{j_{k}=1}^{n}u_{i_k,j_k}\p)\\
&=&\sum\limits_{j_1,\cdots,j_{k-1}=1}^{n}\p u_{i_1,j_1}\cdots u_{i_{k-1},j_k-1}\p\\
&=&\cdots=\p.\\

\end{array}
$$
According to the definition of $\B_s(n)$, it follows that the product $u_{i_1,j_1}\cdots u_{i_k,j_k}$ is not vanishing only if it satisfies that $i_t\neq i_{t+1}$ whenever $j_t\neq j_{t+1}$ for all $1\leq t\leq k-1$.

Given a set $S$, a collection of disjoint nonempty sets $P=\{V_i|i\in I\}$ is called a partition of $S$ if $\bigcup\limits_{i\in I}V_i=S$, $V_i\in P$ is called a block of the partition $P$. Let $S$ be a finite ordered set, then all the partitions of $S$ have finite blocks. A partition $P=\{V_1,\cdots V_r\}$ of $S$ is interval if there are no two distinct blocks $V_i$ and $V_j$ and elements $a, c\in V_i$ and $b, d\in V_j$ s.t. $a < b < c$ or $b<c < d$. An interval partition $P=\{W_s|1\leq s\leq r\}$ is ordered if  $a<b$ for all $a\in W_s$, $b\in W_t$ and $s<t$. We denote by $P_I(S)$ the collection of ordered interval partitions of $S$.

Let $I$ be an index set, $[k]=\{1,\cdots,k\}$ is an ordered set with the natural order. Let $I^k=I\times I\times\cdots\times I$ be the $k$-fold Cartesian product of the index set $I$. A sequence of indices $(i_m)_{m=1,\cdots,k}\in I^k $ is said to be compatible with an ordered interval partition $P=\{W_1,\cdots,W_r\}\in P_I([k])$ if $i_a=i_b$ whenever $a,b$ are in the same block and $i_a\neq i_b$ whenever $a,b$ are in two consecutive blocks, i.e. $W_s$ and $W_{s+1}$ for some $1\leq s\leq r$.  One should pay attention that $i_a=i_b$ is allowed for $a\in W_s$ and $b\in W_{s+2}$ for some $1\leq s\leq r$ 

Now, we define an equivalent relation $\sim_{P_I([k])}$ on $I^k$: two sequences of indices $$(i_m)_{m=1,\cdots,k}\sim_{P_I([k])}(j_m)_{m=1,\cdots,k}$$ if the two sequences are both compatible with an ordered interval partition $P\in P_I([k])$.\\

Let $\J=(i_m)_{m=1,\cdots,k},\J'=(j_m)_{m=1,\cdots,k}\in \{1,...,n\}^k$, we denote $\p u_{i_1,j_1}u_{i_2,j_2}\cdots u_{i_k,j_k}\p$ by $U_{\J,\J'}$.

\begin{lemma}
Fix $k\in \mathbb{N}$, let $\B_s(n)$ be the boolean permutation quantum  semigroup  with standard generators $\{u_{i,j}\}_{i,j=1,\cdots,n}$ and $\p$. Let $\J_1=(i_1,\cdots,i_k),\J_2=(j_1,\cdots,j_k)\in[n]^k$ be two sequences if indices. Then, the product $U_{J_1,\J_2}$ is not vanishing if $\J_1\sim_{P_I([k])}\J_2$
\end{lemma}
\begin{proof}
Suppose $\J_i$ is compatible with an ordered interval partition $P_i$ for $i=1,2$. Let $P_1=\{W_1,\cdots, W_{r_1}\}$ and $P_2=\{W'_1,\cdots,W'_{r_2}\}$, then $P_1\neq P_2$ implies that there exists a $t$ such that $W_t\neq W'_t$ for some $1\leq t\leq\min\{r_1,r_2\}$. Take the smallest $t$, then $W_s=W'_s$ whenever $s<t$ and $W_t\neq W'_t$. Then, these two intervals begin with the same number but end with different numbers, in other words , we have either $W_t\subsetneqq W'_t$ or $W'_t\subsetneqq W_t$. Without loss of generality, we assume $W_t\subsetneqq W'_t$, then there is a number $q$ s.t $q\in W_t$ but $q+1\not\in W_t$ and $q,q+1\in W_t'$. Now, we have $i_q\neq i_{q+1}$ and $j_{q}=j_{q+1}$,  thus
$$U_{\J_1,\J_2}=\p u_{i_1,j_1}\cdots u_{i_q,j_q}u_{i_{q+1},j_{q+1}}\cdots u_{i_k,j_k}\p=0.$$
\end{proof}

\begin{lemma} Let $(\A,\B,E:\A\rightarrow\B)$ be an operator valued probability space. Let $(x_i)_{i=1,...,n}$ be a sequence of $n$ random variables  which are identically distributed  and boolean independent with respect to $E$. Given two sequences of indices $\J=(i_q)_{q=1,\cdots,k}$, $\J'=(j_q)_{q=1,\cdots,k}\in [n]^k$ and $\J\sim_{P_I([n])}\J'$, then $$E[x_{i_1}b_1x_{i_2}b_2\cdots b_{k-1}x_{i_k})]=E[x_{j_1}b_1x_{j_2}b_2\cdots b_{k-1}x_{j_k}],$$ where $b_1,\cdots,b_{k-1}\in \B\cup\{I_{\A}\}$.
\end{lemma}
\begin{proof}
Suppose that $\J$ and $\J'$ are compatible with an ordered interval partition $P=\{W_1,\cdots,W_r\}$. Assume that $W_1=\{1,\cdots,k_1\}$, $W_2\{k_1+1,\cdots, k_2\}$,...,$W_r=\{k_{r-1}+1,\cdots,k)\}$, then $i_{k_t}\neq i_{k_t+1}$ and $j_{k_t}\neq j_{k_t+1}$ for $t=1,...,r$. For convenience, we let $k_r=k$, $k_0=0$ and $b_k=I_{\A}$, we have

$$
\begin{array}{rcl}
& &E[x_{i_1}b_1x_{i_2}b_2\cdots b_{k-1}x_{i_k}]\\
&=&E[x_{i_1}b_1x_{i_2}b_2\cdots b_{n-1}x_{i_k}b_k]\\
&=&E[\prod\limits_{s=1}^r(\prod\limits_{t=n_{s-1}+1}^{n_s}x_{i_t}b_t)]\\
&=&\prod\limits_{s=1}^r E[\prod\limits_{t=n_{s-1}+1}^{n_s}x_{i_t}b_t)]\\
&=&\prod\limits_{s=1}^r E[\prod\limits_{t=n_{s-1}+1}^{n_s}x_{j_t}b_t]\\
&=&E[\prod\limits_{s=1}^r\prod\limits_{t=n_{s-1}+1}^{n_s}x_{j_t}b_t]\\
&=&E[x_{j_1}b_1x_{j_2}b_2\cdots b_{k-1}x_{j_k}].\\
\end{array}$$

\end{proof}

We will write $\sim_{P_I}$ short for $\sim_{P_I([k])}$ when there is no confusion.

\begin{theorem}\label{4.5} Let $(\A,\B,E:\A\rightarrow\B)$ be an operator valued probability space, $\A$ be unital and $\{x_i\}_{i=1,...,n}$  be a sequence of $n$ random variables in $\A$ which is identically distributed  and boolean independent with respect to $E$. Let  $\phi$ be a linear functional on $\B$ and $\bar\phi$ is a linear functional on $\A$ where $\bar\phi(\cdot)=\phi(E[\cdot])$.  Then, the joint distribution of the sequence $\{x_i\}_{i=1,...,n}$ with respect to $\bar \phi$ is invariant under the linear coaction of the boolean permutation quantum semigroup $\B_s(n)$.
\end{theorem}
\begin{proof} Fix $k \in \mathbb{N}$, and  indices $1\leq i_1,\cdots, i_k\leq n$, and $b_1,\cdots,b_{k-1}\in \B\cup \{I_{\A}\}$, where $I_{\A}$ is the unit of $\A$, by the two lemmas above we have

$$
\begin{array}{rcl}
& &\sum\limits_{j_1,j_2,\cdots,j_k=1}^n E[x_{j_1}b_1x_{j_2}b_2\cdots b_{k-1}x_{j_k}]\otimes \p u_{i_1,j_1}\cdots u_{i_k,j_k}{\bf P}\\
&=&\sum\limits_{\substack{j_1,j_2,\cdots,j_k=1\\ (j_s)_{s=1,...,k}\sim_{P_I}(i_t)_{t=1,...,k}}}^n E[x_{j_1}b_1x_{j_2}b_2\cdots b_{k-1}x_{j_k}]\otimes \p u_{i_1,j_1}\cdots u_{i_k,j_k}{\bf P}\\
&=&\sum\limits_{\substack{j_1,j_2,\cdots,j_k=1\\ (j_s)_{s=1,...,k}\sim_{P_I}(i_t)_{t=1,...,k}}}^n E[x_{i_1}b_1x_{i_2}b_2\cdots b_{k-1}x_{i_k}]\otimes \p u_{i_1,j_1}\cdots u_{i_k,j_k}{\bf P}\\
&=&\sum\limits_{j_1,j_2,\cdots,j_n=1}^k E[x_{i_1}b_1x_{i_2}b_2\cdots b_{k-1}x_{i_k}]\otimes \p u_{i_1,j_1}\cdots u_{i_k,j_k}{\bf P}\\
&=&E[x_{i_1}b_1x_{i_2}b_2\cdots b_{k-1}x_{i_k}]\otimes{\bf P}.
\end{array}
$$

Let $b_1,...,b_{k-1}=1_{\A}$ and  let $\phi\otimes id_{\B_s(n)}$ act on the two sides of the above equation then we have 
$$
\begin{array}{rcl}

& &\bar\phi(x_{i_1}x_{i_2}\cdots x_{i_k}){\bf P}\\
&=&\bar\phi(x_{i_1}x_{i_2}\cdots x_{i_k}){\bf P}\\
&=&\sum\limits_{j_1,j_2,\cdots,j_k=1}^n \bar\phi(x_{j_1}x_{j_2}\cdots x_{j_n})\p u_{i_1,j_1}\cdots u_{i_k,j_k}{\bf P},\\
\end{array}
$$
which is our desired conclusion.
\end{proof}

\section{Properties of Tail Algebra for Boolean Independence}
In order to study  boolean exchangeable sequences of random variables, we need to choose a suitable kind of noncommutative probability spaces. It is pointed by Hasebe \cite{Ha} that the $W^*$-probability with faithful normal states does not contain boolean independent random variables with Bernoulli law. Therefore, in our work,I it is necessary to consider $W^*$ probability spaces with more general states rather than faithful states:
\begin{definition} Let  $\A$ be a  von Neumann algebra, a normal state $\phi$ on $\A$ is said to be non-degenerated if  $x=0$ whenever $\phi(axb)=0$ for all $a,b\in \A.$  
\end{definition}

\begin{remark} By proposition 7.1.15 in \cite{KR},  if $\phi$ is a non-degenerated normal state on $\A$ then the GNS representation associated to $\phi$ is faithful. A  faithful normal state on $\A$ is  faithful on all $\A$'s subalgebras but a non-degenerated normal state on $\A$ may not be necessarily  non-degenerated on $\A$'s subalgebras. 
\end{remark}

Let $(\A,\phi)$ be a $W^*$-probability space with a non-degenerated normal state $\phi$. Suppose $\A$ is generated by an infinite sequence of random variables $\{x_i\}_{i\in\mathbb{N}}$, whose joint distribution is invariant under the linear coaction of the quantum semigroups $\B_s(n)$. Let $\A_0$ be the non-unital algebra over $\C$ generated by $\{x_i\}_{i\in\mathbb{N}}$. In this section, we assume that the unit $1_{\A}$ of $\A$ is contained in the weak closure of $\A_0$. We will denote the GNS construction associated to $\phi$ by $(\HH,\xi,\pi)$, then there is a linear map $\hat\cdot: \A_0\rightarrow \HH$ such that $\hat a=\pi(a)\xi$ for all $a\in{\A_0}$.  In the usual sense, the tail algebra $\AT$ of  $\{x_i\}_{i\in\mathbb{N}}$ is defined by:
$$\AT=\bigcap\limits_{n=1}^\infty vN\{x_k|k\geq n\},$$
where $vN\{x_k|k\geq n\}$ is the von Neumann  algebra generated by $\{x_k|k\geq n\}$. We will call $\AT$ unital tail algebra in this paper. In this section, the range algebra we  use is a "non-unital tail algebra" $\T$. The non-unital tail algebra  $\T$ of $\{x_i\}_{i\in\mathbb{N}}$ is given by the follows:
$$\T=\bigcap\limits_{n=1}^\infty W^*\{x_k|k\geq n\},$$
where $W^*\{x_k|k\geq n\}$ is the WOT closure of the non-unital  algebra generated by $\{x_k|k\geq n\}$.  If the unit of $\A$ is contained in $\T$, then $\T$ is also the unital tail-algebra of $\{x_i\}_{i\in\mathbb{N}}$. For convenience, we denote $\A_n$ by the non-unital algebra generated by $\{x_k|k> n\}$. Now, we turn to define our $T$-linear map, the method comes from \cite{KS}. Because we are dealing with von Neumann algebras with non-degenerated normal states which are more general than the faithful states, it is necessary to provide a complete construction here. In \cite{Ko},the normal conditional expectation K\"{o}stler constructed via the shift of the random variables requires the sequence only to be spreadable.  But in our situation, the existence of the normal linear map relies on the invariance under the quantum semigroups $\B_s(n)$'s.\\

\begin{lemma}\label{E}
Let $\A$ be a von Neumann algebra generated by an infinite sequence of selfadjoint random variables $(x_i)_{i\in\mathbb{N}}$, $\phi$ be a non-degenerated normal state on $\A$. If the sequence $(x_i)_{i\in\mathbb{N}}$ is exchangeable in $(\A,\phi)$, then there is a $C^*-$isomorphism  $\alpha: \A_0^{\|\cdot\|}\rightarrow \A_1^{\|\cdot\|}$ such that,
$$\alpha(x_i)=x_{i+1},$$
for all $i\in\mathbb{N}$, where $\A_i^{\|\cdot\|}$ is the $C^*-$algebra generated by $\A_i$.
\end{lemma}
\begin{proof}
Let $(\HH,\xi,\pi)$ be the GNS construction associated to $\phi$, it follows that $\{\hat a|a\in \A_0\}$ is dense in $\HH$. For each $n\in\mathbb{N}$, denote by $A_{[n]}$ the non-unital algebra generated by $\{x_i|i\leq n\}$. Then $\bigcup\limits_{n=1}^{\infty}\{\pi(a)\xi|a\in A_{[n]}\}$ is dense in $\HH$. 
Given $y\in\bigcup\limits_{n=1}^{\infty} A_{[n]}$, there exists $N\in\mathbb{N}$ such that $y\in A_{[N]}$. We can assume $y=p(x_1,...,x_N)$ for some $p\in\C\langle X_1,...,X_N\rangle_0$, then we have 
$$
\begin{array}{rcl}
&&\|\pi(p(x_1,...,x_N))\xi\|^2=\phi (\pi(p(x_1,...,x_N)^*(p(x_1,...,x_N)))\\
&=&\phi (p(x_2,...,x_{N+1})^*(p(x_2,...,x_{N+1}))\\
&=&\|\pi(p(x_2,...,x_{N+1}))\xi\|^2
\end{array}
$$

We can define an isometry $U$ from $\HH$ to its subspace $\HH_1$ which is generated by $\{\hat{a}|a\in A_{1}\}$ by the following formula:
$$U\pi(x_{i_1}\cdots x_{i_k})\xi=\pi(x_{i_1+1}\cdots x_{i_k+1})\xi,$$
for all $i_1,...,i_k\in\mathbb{N}$.\\
Since  $\phi$ gives a faithful representation to $\A$, it gives a faithful representation to $\A_0^{\|\cdot\|}$. For all $y\in\A_1$, according to the faithfulness, we have 
$$\|y\|^2=\sup\{\frac{\langle y^*y\hat a,\hat a \rangle}{\langle \hat a,\hat a \rangle}|a\in\A_0, \hat a\neq 0\}=\sup\{\frac{\phi(a^*y^*ya)}{\phi(a^*a)}|a\in\A_0, \phi(a^*a)\neq 0\}.$$

Denote by $(\HH',\xi',\pi')$ the GNS representation of $\A_1$ associated to $\phi$. Indeed, $\HH'$ can treated as $\HH_1$. Because the  identity of $\A$ is contained in the weak$^*$-closure  of the non unital algebra generated by $(x_i)_{i\in\mathbb{N}}$, by the Kaplansky density theorem, there exists a bounded sequence $\{y_i|\|y_i\|\leq 1\}\in \bigcup\limits_{n=1}^{\infty}A_{[n]}$ such that $y_i$ converges to $1_{\A}$ in WOT. Therefore, $\pi(y_i)\xi$ converges to $\xi$ in norm. Again, by the exchangeability of $(x_i)_{i\in\mathbb{N}}$ and $U\pi(y_i)\xi\in\{\hat b|b\in\A_1\}$ for all $i$,  we have 
$$\|U\pi(y_i)\xi\|=\|\pi(y_i)\xi\|\leq 1$$
and $$\langle U\pi(y_i)\xi,\xi\rangle=\langle\pi(y_i)\xi,\xi\rangle\rightarrow 1. $$
Therefore, $U\pi(y_i)\xi$ converges to $\xi$ in norm, namely, $\xi\in\HH_1.$

Let $x\in\A_1$, then $x=p(x_2,...x_{N+1})$ for some $N$ and $p\in\C\langle X_1,...,X_N\rangle_0$. For every $y\in\A_0$ there exists an $M$, such that $y=p'(x_1,...,x_M)$ for some $p'\in\C\langle X_1,...,X_M\rangle_0$. By the exchangeability, we send $x_1$ to $x_{N+M}$. Then 
$$
\begin{array}{rcl}
\|\pi(x)\hat{y}\|_{\HH}^2&=&\phi(p'(x_1,...,x_M)^*p(x_2,...x_{N+1})^*p(x_2,...x_{N+1})p'(x_1,...,x_M))\\
&=&\phi(p'(x_{M+N},...,x_M)^*p(x_2,...x_{N+1})^*p(x_2,...x_{N+1})p'(x_{N+M},x_2,x_3...,x_M))\\
&=&\|\pi'(x)\widehat{p'(x_{M+N},x_2...,x_M))}\|_{\HH'}^2
\end{array}
$$
and $$\|\widehat{p'(x_1,...,x_M))}\|_{\HH}=\|\widehat{p'(x_{M+N},x_2...,x_M))}\|_{\HH'}.$$
Therefore, we get 
$$\{\frac{\|\pi(x)\hat{a}\|_{\HH}}{\|\hat{a}\|_{\HH}}|a\in\A_0, \hat a\neq 0\}\subseteq\{\frac{\|\pi'(x)\hat{a}\|_{\HH'}}{\|\hat{a}\|_{\HH'}}|a\in\A_1, \hat a\neq 0\},$$ which implies
$$\|x\|=\|\pi(x)\|=\sup\{\frac{\|\pi{x}\hat{a}\|_{\HH}}{\|\hat{a}\|_{\HH}}|a\in\A_0, \hat a\neq 0\}\leq\sup\{\frac{\|\pi'(x)\hat{a}\|_{\HH'}}{\|\hat{a}\|_{\HH'}}|a\in\A_1, \hat a\neq 0\}=\|\pi'(x)\|.$$ 
It follows that $\|x\|=\|\pi'(x)\|$ for all $x\in\A_1$. By taking  the norm limit, we have $\|x\|=\|\pi'(x)\|$ for all $x\in \A_1^{\|\cdot\|}$, so the GNS representation of $\A_1^{\|\cdot\|}$ associated to $\phi$ is faithful. \\
Now, we turn to define our $C^*$-isomorphism $\alpha$:\\
Since $U$ is an isometric isomorphism from $\HH$ to $\HH'$, we define a homomorphism $\alpha': \pi(\A_0)\rightarrow B(\HH')$ by the following formula
$$\alpha'(y)=UyU^*,$$
for $y\in \pi(\A_0)$. Let $y\in\pi(\A_{[n]})$, then $y=\pi(p(x_1,...,x_n))$ for some $p\in\C\langle X_1,...,X_n\rangle_0$. For all $v\in\bigcup\limits_{n=2}^{\infty}\{\pi(a)\xi|a\in A_{[n]}\subset\HH'$, there exists $N\in \mathbb{N}$ and $p_1\in\C\langle X_1,...X_N\rangle_0$  such that $v=\pi(p_1(x_2,...,x_{N+1}))\xi$. We have
$$
\begin{array}{rcl}
\alpha'(y)v&=&U\pi(p(x_1,...,x_n)U^*\pi(p_1(x_2,...,x_{N+1}))\xi\\
&=&U\pi(p(x_1,...,x_n)\pi(p_1(x_1,...,x_{N}))\xi\\
&=&U\pi(p(x_1,...,x_n)p_1(x_1,...,x_{N}))\xi\\
&=&\pi(p(x_2,...,x_{n+1})p_1(x_1,...,x_{N+1}))\xi
\end{array}.
$$
Since $\bigcup\limits_{n=2}^{\infty}\{\pi(a)\xi|a\in A_{[n]}$ is dense in $\HH_1$, we get $\alpha'(\pi(p(x_1,...,x_n))=\pi(p(x_2,...,x_{n+1}))$.  Because $(\HH,\xi,\pi)$ and $(\HH',\xi',\pi')$ are faithful GNS representations for $\A_0$ and $\A_1$ respectively, there is a well defined norm preserving homomorphism $\alpha:\A_0\rightarrow \A_1$, such that $\alpha(x_i)=x_{i+1}$ for all $i\in\mathbb{N}$. Therefore, $\alpha$ extends to a $C^*$-isomorphism from $\A_0^{\|\cdot\|}$ to $\A_1^{\|\cdot\|}$.
\end{proof}

Since $W^*\{x_k|k\geq n\}$'s are WOT closed, their intersection is a WOT closed subset of $\A$. Following the proof of proposition 4.2 in \cite{KS}, we have

\begin{lemma}
For each $a\in\A_0$, $\{\alpha^n(a)\}_{n\in\mathbb{N}}$ is a bounded WOT convergent sequence. Therefore, there exists a well defined $\phi$-preserving linear map $E:\A_0\rightarrow \T$ by the following formula:
$$E[a]=w^*-\lim\limits_{n\infty}\alpha^n(a)$$
for $a\in\A_0$
\end{lemma}

\begin{proof}
By lemma \ref{E}, there is a norm preserving endomorphism $\alpha$ of $\A_0$ such that 
   $$\phi\circ\alpha=\phi\,\,\,\,\,\text{and}\,\,\,\alpha(x_i)=x_{i+1}.$$
For $I\subset\mathbb {N}$, denote by $\A_I$  the non-unital algebra generated by $ \{x_i|i\in I\}$. Suppose $a,b,c\in \bigcup\limits_{|I|<\infty}A_I$,  we can assume $a\in \A_I$,$b\in\A_J$ and $c\in\A_K$ for some finite sets $I,J,K\subset\mathbb{N}$. Because $I,J,K$ are finite, there exists  an $N$ such that $(I\cup K)\cap(J+n)=\emptyset,$ for all $n>N$. We infer from the exchangeability that
  $\phi(a\alpha^n(b)c)=\phi(a\alpha^{n+1}(b)c)$ for all $n>N$. This establishes the limit $$\lim\limits_{n\rightarrow \infty} \phi(a\alpha^n(b)c)$$ 
on the weak$^*$-dense algebra $\bigcup\limits_{|I|<\infty}A_I$. We conclude from this and $\{\alpha^n(b)\}_{n\in\mathbb{N}}$ is bounded  that the pointwise limit of the sequence $\alpha$ defines a linear map $E:\A_0\rightarrow \A$ such that $E(\A_0)\subset \T.$\\
\end{proof}

To extend $E$ to the $W^*-$algebra $\A$, we need to make use of the  boolean invariance conditions. 

\begin{lemma}\label{6.3} Let $(\A,\phi)$ be a noncommutative probability space, $\{x_i\}_{i\in\mathbb{N}}\subset \A$ be an infinite sequence of random variables whose joint distribution is invariant under the linear coactions of the quantum semigroups $\B_s(k)$'s, then 
$$\phi(x_{i_1}^{k_1}x_{i_2}^{k_2}\cdots x_{i_n}^{k_n})=\phi(x_{1}^{k_1}x_{2}^{k_2}\cdots x_{n}^{k_n}),$$
whenever $i_1\neq i_2\neq\cdots \neq i_n,$ and $k_1,...,k_n\in\mathbb{N}$
\end{lemma}
\begin{proof}
If $i_l\neq i_m$ for all $l\neq m$, then the statement holds by the exchangeability of the sequence. Suppose the number $i_l$ appears $m$ times in the sequence, which are $\{i_{l                                                                                                                                          _j}\}_j=1,...,m$ such that $i_{l_j}=i_{l}$ and $l_1< l_2<\cdots<l_m$. Since the sequence is finite, with out losing generality, we can assume that $i_1,...,i_n\leq N+1$ and $i_{l_j}=N+1$ for some $N$  by the exchangeability.\\

For each $M\in\mathbb{N}$, by lemma 4.2, we have the following representation $\pi_M$ of the quantum semigroup $\B_s(M+N)$: 
$$\pi_M(u_{i,j})=\left\{
\begin{array}{lr}
P_{i-N,j-N},\,\,\,\,\,\,\,&\text{if}\,\,\,\, \min\{i,j\}>N \\
\delta_{i,j}P,\,\,\,\,\,\,\,&\text{if}\,\,\,\,\min\{i,j\}\leq N \\
\end{array}
\right.,
$$
and $\pi(\p)=P$, where $p_{i,j}$ and $p$ are projections in $B(\C^{2M})$given by lemma 4.2. Then we have 
$$PP_{i,j}P=\frac{1}{M}P,$$
for $1\leq i,j\leq N$.\\
According to the boolean invariance condition, we have:

$$
\begin{array}{rcl}
& &\phi(x_{i_1}^{k_1}x_{i_2}^{k_2}\cdots x_{i_n}^{k_n})P\\

&=&\sum\limits_{j_1,j_2,...j_n=1}^{M+N} \phi(x_{i_1}^{k_1}x_{i_2}^{k_2}\cdots x_{i_n}^{k_n}) Pu_{j_1,i_1}\cdots u_{j_n,i_n}P \\

&=&\sum\limits_{j_{l_1},j_{l_2},...j_{l_m}=1}^{N} \phi(x_{i_1}^{k_1}\cdots x_{j_{l_1}}^{k_{l_1}}\cdots x_{j_{l_2}}^{k_{l_2}}\cdots x_{i_n}^{k_n}) PP_{j_{l_1},i_{l_1}}PP_{j_{l_2},i_{l_2}}P\cdots u_{j_{l_m},i_{l_m}}P \\

&=&\frac{1}{M^m}\sum\limits_{j_{l_1},j_{l_2},...j_{l_m}=1}^{N} \phi(x_{i_1}^{k_1}\cdots x_{j_{l_1}}^{k_{l_1}}\cdots x_{j_{l_2}}^{k_{l_2}}\cdots x_{i_n}^{k_n})P \\

&=&\frac{1}{M^m}[\sum\limits_{j_{l_s}\neq j_{l_t}\, \text{if}\, s\neq t}^{N} \phi(x_{i_1}^{k_1}\cdots x_{j_{l_1}}^{k_{l_1}}\cdots x_{j_{l_2}}^{k_{l_2}}\cdots x_{i_n}^{k_n})P + \sum\limits_{j_{l_s}=j_{l_t}\, \text{for some}\, s\neq t}^{N} \phi(x_{i_1}^{k_1}\cdots x_{j_{l_1}}^{k_{l_1}}\cdots x_{j_{l_2}}^{k_{l_2}}\cdots x_{i_n}^{k_n})P].\\
\end{array}
$$
In the first part of the sum, by the exchangeability, it follows that

$$\phi(x_{i_1}^{k_1}\cdots x_{j_{l_1}}^{k_{l_1}}\cdots x_{j_{l_2}}^{k_{l_2}}\cdots x_{i_n}^{k_n})=\phi(x_{i_1}^{k_1}\cdots x_{N+1}^{k_{l_1}}\cdots x_{N+2}^{k_{l_2}}\cdots x_{i_n}^{k_n}),$$
where we sent $j_{l_s}$ to $N+s$. Then, we have 
$$
\frac{1}{M^m}\sum\limits_{j_{l_s}\neq j_{l_t}\, \text{if}\, s\neq t}^{N} \phi(x_{i_1}^{k_1}\cdots x_{j_{l_1}}^{k_{l_1}}\cdots x_{j_{l_2}}^{k_{l_2}}\cdots x_{i_n}^{k_n})P=
\frac{\prod\limits_{s=0}^{m-1}(M-s)}{M^m}\phi(x_{i_1}^{k_1}\cdots x_{N+1}^{k_{l_1}}\cdots x_{N+2}^{k_{l_2}}\cdots x_{i_n}^{k_n})P,
$$
which converges to $\phi(x_{i_1}^{k_1}\cdots x_{N+1}^{k_{l_1}}\cdots x_{N+2}^{k_{l_2}}\cdots x_{i_n}^{k_n})P$ as $M$ goes to $\infty$.\\
To the second part of the sum, we have 
$$\phi(x_{i_1}^{k_1}\cdots x_{j_{l_1}}^{k_{l_1}}\cdots x_{j_{l_2}}^{k_{l_2}}\cdots x_{i_n}^{k_n})\leq\|x_{i_1}^{k_1}\cdots x_{j_{l_1}}^{k_{l_1}}\cdots x_{j_{l_2}}^{k_{l_2}}\cdots x_{i_n}^{k_n}\|\leq\|x_1^{k_1+\cdots+k_n}\|,$$
which is bounded, therefore,
$$
|\frac{1}{M^m}\sum\limits_{j_{l_s}=j_{l_t}\, \text{for some}\, s\neq t}^{N} \phi(x_{i_1}^{k_1}\cdots x_{j_{l_1}}^{k_{l_1}}\cdots x_{j_{l_2}}^{k_{l_2}}\cdots x_{i_n}^{k_n})|\leq (1-\frac{\prod\limits_{s=0}^{m-1}(M-s)}{M^m})\|x_1^{k_1+\cdots+k_n}\|
$$
goes to $0$ as $M$ goes to $\infty$.
By now, we have showed that if there are indices $i_s=i_t$ for $s\neq t$ in the the sequence, we can, with out changing the value of the mixed moments, change them to two different large numbers $j_s,j_t$ such that $j_s$, $j_t$ differ the other indices. After a finite steps, we will have
$$\phi(x_{i_1}^{k_1}x_{i_2}^{k_2}\cdots x_{i_n}^{k_n})=\phi(x_{j_1}^{k_1}x_{j_2}^{k_2}\cdots x_{j_n}^{k_n}),$$
such that all the $j_l$' are not equal to any of the other indices. By the exchangeability, the proof is complete. 
\end{proof}
\begin{corollary}\label{c63}
Let $\{x_i\}_{i\in\mathbb{N}}\subset (\A,\phi)$ be an infinite sequence of random variables whose joint distribution is invariant under the linear coactions of the quantum semigroups $\B_s(k)$'s, then 
$$\phi(x_{i_1}^{k_1}x_{i_2}^{k_2}\cdots x_{i_n}^{k_n})= \phi(x_{j_1}^{k_1}x_{j_2}^{k_2}\cdots x_{j_n}^{k_n}),$$
whenever $i_1\neq i_2\neq\cdots \neq i_n,$, $j_1\neq j_2\neq\cdots \neq j_n,$ $k_1,...,k_n,j_1,...j_n\in\mathbb{N}$. Moreover, we have 
$$\phi(ax_{i_1}^{k_1}x_{i_2}^{k_2}\cdots x_{i_n}^{k_n}b)= \phi(ax_{j_1}^{k_1}x_{j_2}^{k_2}\cdots x_{j_n}^{k_n}b),$$ 
whenever
$i_1\neq i_2\neq\cdots \neq i_n$, $j_1\neq j_2\neq\cdots \neq j_n$, $k_1,...,k_n$, $j_1,...j_n>M$ and $a,b\in \A_{[M]}$ for some $M\in\mathbb{N}$.
\end{corollary}

\begin{lemma} \label{6.4}
For all $a,b,y\in\A_0$, we have $$\langle E(y)\hat a, \hat b \rangle =\langle y\widehat{E(a)}, \widehat{E[b]}\rangle.$$
\end{lemma}
\begin{proof}
Because an element in $\A_0$ is a finite linear combination of the noncommutative monomials, it suffices to show the property in the case:  $b^*=x^{r_1}_{i_1}\cdots x^{r_l}_{i_l}$, $y=x^{s_1}_{j_1}\cdots x^{s_m}_{j_m}$, $a=x^{t_1}_{k_1}\cdots x^{t_n}_{k_n}$, where $i_1\neq i_2\neq \cdots\neq i_l$, $j_1\neq...\neq j_m$, $k_1\neq...\neq k_n$ and  all the power indices are positive integers. Let $N=\max\{i_1,...,i_l,j_1,...,j_m,k_1,...,k_n\}$,
for all $L>N$, we have $i_l\neq j_1+L$ and $j_m+L\neq k_1$. Therefore, we have 
$$
\begin{array}{rcl}
\langle E(y)\hat a, \hat b \rangle&=& \lim\limits_{M\rightarrow \infty}\langle \alpha^M(y)\hat a, \hat b \rangle\\
&=& \langle \alpha^L(y)\hat a, \hat b \rangle\\
&=&\phi(x^{r_1}_{i_1}\cdots x^{r_1}_{i_l}x^{s_1}_{j_1+L}\cdots x^{s_m}_{j_m+L}x^{t_1}_{k_1}\cdots x^{t_n}_{k_n}),\\
\end{array}
$$

by  corollary \ref{c63},
$$
\begin{array}{rcl}
&=&\phi(x^{r_1}_{1}\cdots x^{r_l}_{l}x^{s_1}_{l+1}\cdots x^{s_m}_{l+m}x^{t_1}_{l+m+1}\cdots x^{t_n}_{l+m+n} )\\
&=&\phi(x^{r_1}_{1}\cdots x^{r_l}_{l}x^{s_1}_{l+1}\cdots x^{s_m}_{l+m}x^{t_1}_{l+m+1}\cdots x^{t_n}_{l+m+n} )\\
&=&\phi(x^{r_1}_{i_1+L}\cdots x^{r_1}_{i_l+L}x^{s_1}_{j_1}\cdots x^{s_m}_{j_m}x^{t_1}_{k_1+2L}\cdots x^{t_n}_{k_n+2L})\\
&=&\phi(\alpha^L(x^{r_1}_{i_1}\cdots x^{r_1}_{i_l})x^{s_1}_{j_1}\cdots x^{s_m}_{j_m}\alpha^{2L}(x^{t_1}_{k_1}\cdots x^{t_n}_{k_n}))\\
&=&\lim\limits_{M\rightarrow \infty}\phi(\alpha^N(b^*)y\alpha^{2L+M}(a))\\
&=& \phi(\alpha^L(b^*)yE[a]).\\
\end{array}
$$
Notice that $\{\alpha^L(b)|L\leq N\}$ is a bounded sequence of random variables which converges to $E[b^*]$ in WOT and $\phi(\cdot yE[a])$ is a normal linear functional on $\A$, we have 
$$
\begin{array}{rcl}
\phi(\alpha^L(b^*)yE[a])&=&\lim\limits_{M\rightarrow \infty}\phi(\alpha^M(b^*)yE[a])\\
&=&\phi(E[b]^*yE[a])\\
&=&\langle y\widehat{E[a]}, \widehat{E[b]}\rangle.\\
\end{array}
$$

\end{proof}

\begin{lemma} 
Let $\{y_n\}_{n\in\mathbb{N}}\subset \A_0$ be a bounded sequence of random variables such that $w^*-\lim y_n=0$, then $w^*-\lim E[y_n]=0.$
\end{lemma}
\begin{proof}
For all $a, b\in\A_0$, we have 
$$
\lim\limits_{n}\langle E[y_n]\hat a,\widehat{E[b]}\rangle=\lim\limits_{n}\langle y_n\widehat{E[a]},\widehat{E[b]} \rangle=0.
$$
Since $\{\hat a|a\in\A_0\} $ is dense in $\HH_\xi$, we get our desired conclusion.
\end{proof}

Let $y\in\A$ and $\{y_n\}_{n\in\mathbb{N}}\subset \A_0$ be a bounded sequence such that $y_n$ converges to $y$ in WOT. For all $a, b\in\A_0$, we have 
$$
\lim\limits_{n}\langle E[y_n]\hat a,\hat b \rangle=\lim\limits_{n}\langle y_n\widehat{E[a]},\widehat{E[b]} \rangle=\langle y\widehat{E[a]},\widehat{E[b]} \rangle.
$$
Therefore, $\{E[y_n]\}_{n\in\mathbb{N}}$ converges to an element $y'$ in pointwise weak topology, by the lemma above, we see that $y'$ is independent of the choice of $\{y_n\}_{n\in\mathbb{N}}$. Since $\{E[y_n]\}_{n\in\mathbb{N}}\subset\T$, we have $y'\in\T$. By  now, we have defined a linear map $E:\A\rightarrow\T$ and  we have

\begin{lemma}
$E$ is normal.
\end{lemma}
\begin{proof}
Let $\{y_n\}_{n\in\mathbb{N}}\subset\A$ be a bounded  WOT convergent sequence of random variables such that $w^*-\lim\limits_{n\rightarrow \infty} y_n= y$. Then, we have 
$$\lim\limits_{n\rightarrow \infty}\langle E[y_n]\hat{a},\hat{b}\rangle=\lim\limits_{n\rightarrow \infty}\langle y_n\widehat{E[a]},\widehat{E[b]}\rangle=\langle y\widehat{E[a]},\widehat{E[b]}\rangle=\langle E[y]\hat{a},\hat{b}\rangle,$$
for all $a,b\in\A_0$. Therefore, $E$ is normal.
\end{proof}
Now, we can turn to show that $E$ is a conditional expectation from $\A$ to $\phi$:
\begin{lemma}\label{6.10} $E[a]=a$ for all $a\in \T$.
\end{lemma}
\begin{proof}
Let $a\in\T$,  $b,c\in\A_0$, then there exists an $N\in\mathbb{N} $ such that $a\in\overline{\A_{N+1}}^{w^*}$ and $b,c\in\A_{[N]}$. We can approximate $a$ in  WOT by a bounded sequence $(a_k)_{k\in\mathbb{N}}\subset \A_{N+1}$ in WOT. According to the definition of $E$ and the exchangeability, we have 
$$
\begin{array}{rcl}
\langle E[a]\hat c,\hat b \rangle
&=&\phi(b^*E[a]c)\\
&=&\lim\limits_k\phi(b^*E[a_k]c)\\
&=&\lim\limits_k\lim\limits_n\phi(b^*\alpha^n(a_k) c)\\
&=&\lim\limits_k\phi(b^*a_kc)\\
&=&\phi(b^*ac)=\langle a\hat c,\hat b \rangle.
\end{array}$$
The equation is true for all $b,c\in\A_0$, so $E[a]=a$.
\end{proof}

To check the bimodule property of $E$, we need to show that the quality of \ref{6.4} holds for all $x\in\A$:
\begin{lemma}\label{6.11}
For all $a,b,x\in\A$, we have  
$$\phi(aE[x]b)=\phi(E[a]xE[b]).$$
\end{lemma}
\begin{proof}
By the Kaplansky's density  theorem, there exist two bounded sequences $\{a_n\in\A_0|\|a_n\|\leq \|a\|, n\in\mathbb{N}\}$ and $\{b_n\in\A_0|\|b_n\|\leq \|b\|, n\in\mathbb{N}\}$ which  converge to $a$ and $b$ in WOT, respectively. Since $\phi$ and $E$ are normal, we have
$$
\begin{array}{rcl}
\phi(aE[x]b)&=& \lim\limits_{n}\phi(a_nE[x]b)\\ 
&=& \lim\limits_{n}\lim\limits_{m}\phi(a_nE[x]b_m)\\ 
&=& \lim\limits_{n}\lim\limits_{m}\phi(E[a_n]xE[b_m])\\
&=& \lim\limits_{n}\phi(E[a_n]xE[b])\\
&=&\phi(E[a]xE[b]).\\
\end{array}
$$
\end{proof}

\begin{lemma}
$E[ax]=aE[x]$ for all $a\in \T$ and $x\in\A$.
\end{lemma}
\begin{proof}
For all $b,c\in\A_{0}$, by lemma \ref{6.11} and Lemma \ref{6.10}, we have
$$
\begin{array}{rcl}
\langle E[ax]\hat b,\hat c\rangle&=& \phi(c^*E[ax]b)\\
&=&\phi(E[c^*]axE[b])\\
&=&\phi((E[c^*]a)xE[b]).
\end{array}
$$
since $E[c^*]a\in\T$, $E[E[c^*]a]=E[c^*]a$, then
$$
\begin{array}{rcl}
\phi((E[c^*]a)xE[b])&=& \phi(E[E[c^*]a]xE[b])\\
&=& \phi(E[c^*]aE[x]b)\\
&=& \phi(E[c^*]E[aE[x]]b)\\
&=& \phi(E[E[c^*]](aE[x]) E[b])\\
&=& \phi(E[c^*](aE[x]) E[b])\\
&=& \phi(c^*E[aE[x]]b)\\
&=& \phi(c^*aE[x]b)\\
&=&\langle aE[x]\hat b,\hat c\rangle.
\end{array}
$$

Since $b,c$ are arbitrary, we get our desired conclusion
\end{proof}

\begin{lemma}\label{6.6}
$$E[x^{k_1}_{i_1}\cdots x^{k_s}_{i_s}\cdots x^{k_t}_{i_t}\cdots x^{k_n}_{i_n}]=E[x^{k_1}_{i_1}\cdots \alpha^N(x^{k_s}_{i_s}\cdots x^{k_t}_{i_t})\cdots x^{k_n}_{i_n}]$$
whenever $ i_1\neq i_2\neq\cdots\neq i_n$, $N\geq\max\{i_1,...,i_n\}$, $k_j$'s are positive integers.
\end{lemma}
\begin{proof} 
Given $a,b\in\A_0$, then there exists an $M$ such that $a,b\in\A_{[M]}$. Then, we have
$$
\begin{array}{lcl}
&&\langle E[x^{k_1}_{i_1}\cdots x^{k_s}_{i_s}\cdots x^{k_t}_{i_t}\cdots x^{k_n}_{i_n}]\hat{a},\hat{b}\rangle\\
&=&\lim\limits_{l\rightarrow \infty}\langle \alpha^l(x^{k_1}_{i_1}\cdots x^{k_s}_{i_s}\cdots x^{k_t}_{i_t}\cdots x^{k_n}_{i_n})\hat{a},\hat{b}\rangle\\
&=&\langle \alpha^M(x^{k_1}_{i_1}\cdots x^{k_s}_{i_s}\cdots x^{k_t}_{i_t}\cdots x^{k_n}_{i_n})\hat{a},\hat{b}\rangle\\
&=&\langle x^{k_1}_{i_1+M}\cdots x^{k_s}_{i_s+M}\cdots x^{k_t}_{i_t+M}\cdots x^{k_n}_{i_n+M}\hat{a},\hat{b}\rangle,\\
\end{array}
$$  
by lemma \ref{c63} and $i_1+M\neq\cdots\neq i_{s-1}+M\neq i_s+M+N\neq i_{s+1}+M+N\neq\cdots\neq i_{t}+M+N\neq i_{t+1}+M\neq \cdots i_{n}+M$,
$$
\begin{array}{lcl}
&&\langle x^{k_1}_{i_1+M}\cdots x^{k_s}_{i_s+M}\cdots x^{k_t}_{i_t+M}\cdots x^{k_n}_{i_n+M}\hat{a},\hat{b}\rangle\\
&=&\langle x^{k_1}_{i_1+M}\cdots x^{k_s}_{i_s+M+N}\cdots x^{k_t}_{i_t+M+N}\cdots x^{k_n}_{i_n+M}\hat{a},\hat{b}\rangle\\
&=&\langle \alpha^M(x^{k_1}_{i_1}\cdots \alpha^N(x^{k_s}_{i_s}\cdots x^{k_t}_{i_t})\cdots x^{k_n}_{i_n})\hat{a},\hat{b}\rangle\\
&=&\lim\limits_{l\rightarrow \infty}\langle \alpha^l(x^{k_1}_{i_1}\cdots \alpha^N(x^{k_s}_{i_s}\cdots x^{k_t}_{i_t})\cdots x^{k_n}_{i_n})\hat{a},\hat{b}\rangle\\
&=&\langle E[x^{k_1}_{i_1}\cdots \alpha^N(x^{k_s}_{i_s}\cdots x^{k_t}_{i_t})\cdots x^{k_n}_{i_n}]\hat{a},\hat{b}\rangle.\\
\end{array}
$$
Because $\{\hat{a}|a\in\A_0\}$ is dense in $\HH$, the proof is complete.
\end{proof}

\begin{corollary}\label{6.7}
$$E[x^{k_1}_{i_1}\cdots x^{k_s}_{i_s}\cdots x^{k_t}_{i_t}\cdots x^{k_n}_{i_n}]=E[x^{k_1}_{i_1}\cdots E[x^{k_s}_{i_s}\cdots x^{k_t}_{i_t}]\cdots x^{k_n}_{i_n}],$$
whenever $ i_1\neq i_2\neq\cdots\neq i_n$.
\end{corollary}
\begin{proof}
Let $N=\max\{i_1,...,i_n\}$. Since $E[x^{k_s}_{i_s}\cdots x^{k_t}_{i_t}]=w^*-\lim\limits_{l\rightarrow\infty} \alpha^l(x^{k_s}_{i_s}\cdots x^{k_t}_{i_t})$, we have
$$E[x^{k_s}_{i_s}\cdots x^{k_t}_{i_t}]=w^*-\lim\limits_{l\rightarrow\infty} \frac{1}{l}\sum\limits_{s=1}^l\alpha^{N+l}(x^{k_s}_{i_s}\cdots x^{k_t}_{i_t}).$$
Then, by lemma \ref{6.6},
$$
\begin{array}{rcl}
&&E[x^{k_1}_{i_1}\cdots x^{k_s}_{i_s}\cdots x^{k_t}_{i_t}\cdots x^{k_n}_{i_n}]\\
&=& \frac{1}{l}\sum\limits_{s=1}^lE[x^{k_1}_{i_1}\cdots \alpha^{N+l}(x^{k_s}_{i_s}\cdots x^{k_t}_{i_t})\cdots x^{k_n}_{i_n}]\\
&=&E[x^{k_1}_{i_1}\cdots [w^*-\lim\limits_{l\rightarrow\infty} \frac{1}{l}\sum\limits_{s=1}^l\alpha^{N+l}(x^{k_s}_{i_s}\cdots x^{k_t}_{i_t})]\cdots x^{k_n}_{i_n}]\\
&=&E[x^{k_1}_{i_1}\cdots E[x^{k_s}_{i_s}\cdots x^{k_t}_{i_t}]\cdots x^{k_n}_{i_n}].\\
\end{array}
$$
The last two equations follow the normality of $E$ and $$x^{k_1}_{i_1}\cdots [\frac{1}{l}\sum\limits_{s=1}^l\alpha^{N+l}(x^{k_s}_{i_s}\cdots x^{k_t}_{i_t})]\cdots x^{k_n}_{i_n} \rightarrow x^{k_1}_{i_1}\cdots E[x^{k_s}_{i_s}\cdots x^{k_t}_{i_t}]\cdots x^{k_n}_{i_n}$$
in WOT.
\end{proof}

\begin{lemma}\label{6.8}
$$E[b_1x^{k_1}_{i_1}b_2\cdots b_sx^{k_s}_{i_s}\cdots b_tx^{k_t}_{i_t}\cdots b_nx^{k_n}_{i_n}]=E[b_1x^{k_1}_{i_1}b_2\cdots E[b_sx^{k_s}_{i_s}\cdots b_tx^{k_t}_{i_t}]\cdots b_nx^{k_n}_{i_n}],$$
whenever $ i_1\neq i_2\neq\cdots\neq i_n$, $k_1,...k_n$ are positive integers, $b_1,...b_n\in \A_{N+1}$ where $N=\max\{i_1,...,i_n\}$.
\end{lemma}
\begin{proof}
By the linearity of $E$, we can assume that $b_i$'s are \lq\lq monomials",  i.e. $b_j=x_{i_{j,1}}\cdots x_{i_{j,r_j}}$ where $i_{j,j'}$'s are greater than $N$. Then,
$$b_1x^{k_1}_{i_1}b_2\cdots b_sx^{k_s}_{i_s}\cdots b_tx^{k_t}_{i_t}\cdots b_nx^{k_n}_{i_n}=b_1x^{k_1}_{i_1}b_2\cdots x_{i_{s,1}}\cdots x_{i_{s,r_s}}x^{k_s}_{i_s}\cdots x_{i_{t,1}}\cdots x_{i_{t,r_t}}x^{k_t}_{i_t}\cdots b_nx^{k_n}_{i_n},$$
$i_{s,1}\geq N+1>i_{s-1}$ and $i_{t,r_t}\geq N+1>i_{t+1}$. Therefore, by lemma \ref{6.7},
$$
\begin{array}{rcl}
&&E[b_1x^{k_1}_{i_1}b_2\cdots x_{i_{s,1}}\cdots x_{i_{s,r_s}}x^{k_s}_{i_s}\cdots x_{i_{t,1}}\cdots x_{i_{t,r_t}}x^{k_t}_{i_t}\cdots b_nx^{k_n}_{i_n}]\\
&=&E[b_1x^{k_1}_{i_1}b_2\cdots E[x_{i_{s,1}}\cdots x_{i_{s,r_s}}x^{k_s}_{i_s}\cdots x_{i_{t,1}}\cdots x_{i_{t,r_t}}x^{k_t}_{i_t}]\cdots b_nx^{k_n}_{i_n}]\\
&=&E[b_1x^{k_1}_{i_1}b_2\cdots E[b_sx^{k_s}_{i_s}\cdots b_tx^{k_t}_{i_t}]\cdots b_nx^{k_n}_{i_n}].
\end{array}
$$

\end{proof}

\begin{proposition}\label{7.1} Let $(\A,\phi)$ be a $W^*$-probability space and $(x_i)_{i\in\mathbb{N}}$ be a sequence of of selfadjoint random variables in $\A$ whose joint distribution is invariant of under the boolean permutations. Let $E$ be the  conditional expectation onto the non-unital tail algebra $\T$ of the sequence. Then, $E$ has the following factorization property: for all $n,k\in\mathbb{N}$, polynomials $p_1,...,p_n\in\T\langle X_1,...,X_k\rangle_0$ and $i_1,...,i_n\in\{1,...,k\}$, we have
$$ E[p_1(x_{i_1})\cdots p_l(x_{i_m})\cdots p_n(x_{i_n})]=E[p_1(x_{i_1})\cdots E[p_l(x_{i_m})]\cdots p_n(x_{i_n})].$$
\end{proposition}
\begin{proof}
It suffices to prove the statement in the case: $p_1,...,  p_n$ are $\T$-monomials but none of them is an element of $\T$. 
Assume that 
$$  p_i(X)=b_{i,0}X^{t_{i,1}}b_{i,1}X^{t_{i,2}}b_{i,2}\cdots X_{t_i}^{k_i},$$ 
where $b_{i,j}\in \T$ and $t_{i,j}'s$ are positive integers. Let $N=\max\{i_1,...,i_n\}$, then $b_{i,j}\in \T\subset \overline{\A_{N+1}}^{w^*}$. By the Kaplansky theorem, for every $b_{i,j}$, there exists a bounded sequence $\{b_{l,i,j}\}_{\l\in\mathbb{N}}$ such that $b_{l,i,j}$ converges to $b_{i,j}$ in strong operator topology (SOT). Let $p_{n,i}(X)=b_{n,i,0}X^{t_{i,1}}b_{n,i,1}X^{t_{i,2}}b_{n,i,2}\cdots X_{t_i}^{k_i}$, then $p_{l,k}(x_{i_k})$ converges to $p_k(x_{i_k})$ in SOT. By the normality of $E$, we have 
$$E[p_1(x_{i_1})\cdots p_l(x_{i_m})\cdots p_n(x_{i_n})]=w^*-\lim\limits_{l\rightarrow \infty} E[p_{l,1}(x_{i_1})\cdots p_{l,m}(x_{i_m})\cdots p_{l,m}(x_{i_n})].$$ 
By lemma \ref{6.8}, we have
$$E[p_{l,1}(x_{i_1})\cdots p_{l,m}(x_{i_m})\cdots p_{l,m}(x_{i_n})]=E[p_{l,1}(x_{i_1})\cdots E[p_{l,m}(x_{i_m})]\cdots p_{l,m}(x_{i_n})].$$
It follows that $E[p_{l,m}(x_{i_m})]$ converges to $E[p_{m}(x_{i_m})]$ in WOT. Therefore, $p_{l,1}(x_{i_1})\cdots E[p_{l,m}(x_{i_m})]\cdots p_{l,m}(x_{i_n})$ converges to $p_{1}(x_{i_1})\cdots E[p_{m}(x_{i_m})]\cdots p_{m}(x_{i_n})$ in WOT. Now, we have
$$
\begin{array}{rcl}
&&E[p_1(x_{i_1})\cdots p_l(x_{i_m})\cdots p_n(x_{i_n})]\\
&=&w^*-\lim\limits_{l\rightarrow \infty} E[p_{l,1}(x_{i_1})\cdots p_{l,m}(x_{i_m})\cdots p_{l,m}(x_{i_n})]\\
&=&w^*-\lim\limits_{l\rightarrow \infty} E[p_{l,1}(x_{i_1})\cdots E[p_{l,m}(x_{i_m})]\cdots p_{l,m}(x_{i_n})]\\
&=&E[p_{1}(x_{i_1})\cdots E[p_{m}(x_{i_m})]\cdots p_{m}(x_{i_n})],\\
\end{array}
$$
the last equality follows $E$'s WOT continuity.
\end{proof}

\section{Main theorem and examples}
\subsection{non-unital tail algebra case}
Now, we can state and prove our main theorem for the non-unital tail algebra case:
\begin{theorem}
Let $(\A,\phi)$ be a $W^*$-probability space and $(x_i)_{i\in\mathbb{N}}$ be an infinite sequence of selfadjoint random variables. Suppose $\A$ is the WOT closure of the non-unital algebra generated by $(x_i)_{i\in\mathbb{N}}$ and $\phi$ is non-degenerated. Then the following statements are equivalent:
\begin{itemize}
\item[a)] The joint distribution of $(x_i)_{i\in \mathbb{N}}$  satisfies the invariance conditions associated with the linear coactions of the quantum semigroups $\B_s(n)$'s.

\item[b)]  The sequence $(x_i)_{i\in\mathbb{N}}$ is identically distributed and boolean independent with respect to a $\phi-$preserving normal conditional expectation $E$ onto the non-unital tail algebra $\T$ of the sequence $(x_i)_{i\in\mathbb{N}}$
\end{itemize}
\end{theorem}
\begin{proof}
$a)\Rightarrow b)$: By choosing $m=1$ in proposition \ref{7.1}, we have
$$ 
\begin{array}{rcl}
&&E[p_1(x_{i_1})\cdots p_2(x_{i_2})\cdots p_n(x_{i_n})]\\
&=&E[E[p_1(x_{i_1})] p_2(x_{i_2})\cdots p_n(x_{i_n})]\\
&=&E[p_1(x_{i_1})] E[p_2(x_{i_2})\cdots p_n(x_{i_n})]\\
&&\cdots\\
&=&E[p_1(x_{i_1})]E[p_2(x_{i_2})]\cdots E[p_n(x_{i_n})],
\end{array}
$$
whenever $i_1\neq i_2\neq\cdots \neq i_n$, $p_1,...,p_n\in\T\langle X\rangle_0$\\
$b)\Rightarrow a)$ is a special case of theorem \ref{4.5}
\end{proof}

\subsection{Unital tail algebra case}
Let $(\A,\phi)$ be a $W^*$ probability space with a non-degenerated normal state $\phi$ and $(x_i)_{i\in\mathbb{N}}$ be a sequence of selfadjoint random variables. Suppose $\A$ is the WOT closure of the unital algebra generated $(x_i)_{i\in\mathbb{N}}$ and $\phi$ is non-degenerated. Again, we denote by $\A_0$ the non-unital algebra generated by $(x_i)_{i\in\mathbb{N}}$. Let $I_{\A}$ be the unit of $\A$, we have considered the case that $1_{\A}$ is contained in $\overline{\A_0}^{w^*}$. If $I_{\A}$ is not contained in $\overline{\A_0}^{w^*}$, denote by $I_1$ the unit of $\overline{\A_0}^{w^*}$. Then 
$$I_2=I_{\A}-I_1\neq 0$$ 
and 
$$\A=\C I_2\oplus \overline{\A_0}^{w^*}.$$ 
For all $x\in \overline{\A_0}^{w^*}$, we have $$I_2x=(I_{\A}-I_1)x=0.$$ Let $a\in\overline{\A_0}^{w^*}$ such that $\phi(xay)=0$ for all $x,y\in \overline{\A_0}^{w^*}$. For $\bar{x}, \bar{y}\in {\A}$,  there exist two constants $ c_1,c_2\in\C$ and $x,y\in\overline{\A_0}^{w^*}$ such that $x=c_1I_2+x$ and $y=c_2I_2+y$, then 
$$\phi(\bar{x}a\bar{y})=\phi(xab)=0,$$ Since our $\bar{x}, \bar{y}$ are chosen arbitrarily, we have $a=0$. Therefore, $(\overline{\A_0}^{w^*},\frac{1}{\phi(I_1)}\phi)$ is a $W^*-$ probability space with a non-degenerated normal state. Let $\A_{tail }$ be the unital tail algebra of $(x_i)_{i\in\mathbb{N}}$ in $(\A,\phi)$ and $\T$  be the non-unital tail algebra of $(x_i)_{i\in\mathbb{N}}$ in $(\overline{\A_0}^{w^*},\frac{1}{\phi(I_1)}\phi)$. Then, we have 
$$\A_{tail}=\bigcap\limits_{n=1}^\infty vN\{x_k|k\geq n\}=\bigcap\limits_{n=1}^\infty (W^*\{x_k|k\geq n\}+\C I_{\A})=\T+\C I_{\A}.$$

Since $\overline{\A_0}^{w^*}$ is a two-sided ideal of $\A$. For $\bar x\in\AT$, $\bar x=aI_{\A}+x$ for some $x\in\T$ and $a\in\C$. By theorem 7.2, there is a $\phi$ preserving normal conditional expectation $E$ from $\overline{\A_0}^{w^*}$ onto $\T$. As we proceeded  in section 7, we can extend this conditional expectation $E$ to an conditional expectation $\bar E$ which is  from the unitalization of $\overline{\A_0}^{w^*}$ to the unitalization of $\T$. The unitalizations of the two algebras  are  isomorphic to $\A$ and $\AT$, respectively. We have
\begin{lemma}
The conditional expectation $\bar E$ is $\phi$-preserving and normal.
\end{lemma}
\begin{proof}
The normality is obvious, we just check the $\phi$-preserving condition here. Let $\bar x=aI_{\A}+x\in\A$ for some $x\in\overline{\A_0}^{w^*}$ and $a\in\C$,  we have 
$$\phi(E[\bar x]=\phi(E[aI_{\A}+x])=\phi(aI_{\A}+E[x])=a+\phi(E[x])=a+\phi(x).$$
The last equality follows the fact that $E$ is a ${\phi(I_1)}\phi$-preserving conditional expectation in $(\overline{\A_0}^{w^*},\frac{1}{\phi(I_1)}\phi)$.
\end{proof}
Together with proposition \ref{4.5} We have the following theorem for our unital case:

\begin{theorem}
Let $(\A,\phi)$ be a $W^*$-probability space and $(x_i)_{i\in\mathbb{N}}$ be a sequence of selfadjoint random variables. Suppose  the unit $I_{\A}$ of $\A$ is not contained in the WOT closure of the non-unital algebra generated by $(x_i)_{i\in\mathbb{N}}$ and $\phi$ is non-degenerated. Then the following statements are equivalent:
\begin{itemize}
\item[a)] The joint distribution of $(x_i)_{i\in \mathbb{N}}$  satisfies the invariance condition associated with the linear coactions of the quantum semigroups $\B_s(n)$'s.

\item[b)]  The sequence $(x_i)_{i\in\mathbb{N}}$ is identically distributed and boolean independent with respect to a $\phi-$preserving normal conditional expectation $E$ onto the unital tail algebra $\AT$ of the $(x_i)_{i\in\mathbb{N}}$. 
\end{itemize}
\end{theorem}

\subsection{Examples}
To illustrate theorem 7.1 and theorem 7.3, we provide two examples here. For the details of the examples, see \cite{CF} and \cite{Fi}.\\
{\bf Non-unital case} 
Let $\HH$ be a Hilbert space with orthonormal basis $\{e_i\}_{i\in\mathbb{N}\cup \{0\}}$, we define a sequence of operators $\{x_n\}_{n\in\mathbb{N}}$ as follows:
$$x_n e_0=e_n, \text{and}\,\,\ x_ne_i=\delta_{n,i}e_0 \,\,\text{for}\,\,\,i\in\mathbb{N}.$$
Let $\A$ be the von Neumann algebra generated by $\{x_n\}_{n\in\mathbb{N}}$, then $e_0$ is cyclic for $\A$. Since $\A$ is WOT closed and  contains all finite-rank operators, $\A$ is actually $B(\HH)$. Let $\phi$ be the vector state $\phi(\cdot)=\langle \cdot e_0, e_0 \rangle$, then we can easily check that the random variables  $x_i$'s are identically distributed and boolean independent. The tail algebra is  $\C P_{e_0}$ which does not contain the unit of $B(\HH)$. The conditional expectation  $E$ is given by the following formula:
$$E[x]=P_{e_0}xP_{e_0},$$
for all $x\in \A.$\\
{\bf Unital case} 
Let $\HH_1=\HH\oplus \C e_{-1}$ be the direct sum of the Hilbert space $\HH$ with orthonormal basis $\{e_i\}_{i\in\mathbb{N}\cup \{0\}}$ and $\C P_{e_{-1}}$. As we constructed in the previous example, we define a sequence of operators $\{x_n\}_{n\in\mathbb{N}}$ as follows: 
$$x_n e_0=e_n, \text{and}\,\,\ x_ne_i=\delta_{n,i}e_0 \,\,\text{for}\,\,\,i\in\mathbb{N}.$$ Let $\A$ be the von Neumann algebra generated by $\{x_n\}_{n\in\mathbb{N}}$, then $\A=B(\HH)\oplus \C P_{e_{-1}}$. Therefore, the WOT-closure of the non-unital algebra generated by   $\{x_n\}_{n\in\mathbb{N}}$ is $B(\HH)\oplus 0$ but not the entire algebra $\A$. 
Let $\phi$ be the vector state $\phi(\cdot)=\frac{1}{2}\langle \cdot (e_0+e_{-1}), e_0+e_{-1}\rangle$, then  the random variables  $x_i$'s are identically distributed and boolean independent. The unital tail algebra is  $\C I_{\HH}\oplus\C P_{e_0}$ which contains the unit of $B(\HH_1)$. The conditional expectation  $E$ is given by the following formula:
$$E[x]=P_{e_0}xP_{e_0}+\langle x e_{-1}, e_{-1}\rangle (I_{\HH_1}-P_{e_{-1}}),$$
for all $a\in\A$.

\subsection{On $W^*$-probability spaces with faithful states}

If we restrict the invariance condition for boolean independence to a $W^*$-probability space with a faithful state, then we will have the following:
\begin{theorem}
Let $(\A,\phi)$ be a $W^*$-probability space and $(x_i)_{i\in\mathbb{N}}$ be a sequence of selfadjoint random variables such that $\A$ is generated by  $(x_i)_{i\in\mathbb{N}}$ and $\phi$ is faithful. Then the following statements are equivalent:
\begin{itemize}
\item[a)] The joint distribution of $(x_i)_{i\in \mathbb{N}}$  satisfies the invariance condition associated with the linear coactions of the quantum semigroups $\B_s(n)$'s.

\item[b)]  $x_i=x_j$ for all $i,j\in \mathbb{N}$
\end{itemize}
\end{theorem}
\begin{proof} 

$b)\Rightarrow a)$: If  $x_i=x_j$ for all $i,j\in \mathbb{N}$, given a  monomial $p=X_{i_1}\cdots X_{i_k}\in\C\langle X_1,...,X_n\rangle$, then
$$
\begin{array}{rcl}
\mu_{x_1,...x_n}(X_{i_1}\cdots X_{i_k})\p&=&\phi(x_{i_1}\cdots x_{i_k})\p\\
&=&\phi(x_1^k)\p\\
&=& \sum\limits_{j_1,...,j_k=1}^n \phi(x_1^k)\pi(\p u_{j_1,i_1} \cdots u_{j_k,i_k}\p)\\
&=&\sum\limits_{j_1,...,j_k=1}^n \phi(x_{j_1}\cdots x_{j_k})\p u_{j_1,i_1}\cdots {j_k,i_k}\p\\
&=&\mu_{x_1,...x_n}\otimes id_{\B_s(n)}(\L p).\\
\end{array}
$$
$b)\Rightarrow a)$: It is sufficient to show that $x_1=x_2$. By theorem 7.1 and 7.3, there exists a $\phi-$preserving conditional expectation $E$ maps $\A$ to its unital or non-unital tail algebra such that $(x_i)_{i\in\mathbb{N}}$ is identically boolean independent with respect to $E$. For $k\in\mathbb{N}$ and $k>2$, we have 
$$
\begin{array}{rcl}
&&\phi((x_1-x_2)x_k((x_1-x_2)x_k)^*)\\
&=&\phi((x_1-x_2)x^2_k(x_1-x_2))\\
&=&\phi(E[(x_1-x_2)x^2_k(x_1-x_2)])\\
&=&\phi(E[x_1-x_2]E[x^2_k]E[x_1-x_2])\\
&=&0.
\end{array}
$$
Since $\phi$ is faithful, we get 
$$(x_1-x_2)x_k=0$$
 for all $k> 2$. Let $\A_n$ be the WOT closure of the non-unital algebra generated by $\{x_k|k> n\}$, then
 we have $$(x_1-x_2)x=0$$
  for all $x\in \A_k$ . Notice that $(x_i)_{i\in\mathbb{N}}$ is exchangeable, by the construction of proposition 4.2 in \cite{KS} , there exists a normal $\phi$-preserving  homomorphism $\alpha: \A_n\rightarrow \A_{n+1}$ such that $\alpha(x_i)=x_{i+1}$. Denote by $I_n$ the unit of $\A_n$, then $\alpha(I_n)=I_{n+1}$ and  $I_nI_{n+1}=I_{n+1}$, since $I_{n+1}$ is a projection in $\A_n$. Then, we have
$$\phi((I_n-I_{n+1})^2)=\phi(I_n-I_{n+1})=\phi(I_n)-\phi(\alpha(I_n))=0,$$
which implies that $I_n=I_{n+1}$. It follows that $$I_0=I_1=I_2.$$ Therefore,
$$0=(x_1-x_2)I_2=(x_1-x_2)I_0=x_1-x_2.$$

\end{proof}

\section{Two more kinds of distributional symmetries}
Since $\C\langle X_1,...,X_n\rangle$ is an algebra which is freely generated by $n$ indeterminants $X_1,...,X_n$. It would be natural to define  coactions of the quantum semigroups $\B_s(n)$ on $\C\langle X_1,...,X_n\rangle$ as an algebraic homomorphism  but not only a linear map.  In this section, we will study the probabilistic symmetries associated with  some algebraic coactions of the quantum semigroups $\B_s(n)$'s and $B_s(n)$'s on $\C\langle X_1,...,X_n\rangle$. We we will define the invariance condition for the joint distribution of a sequence of noncommutative random variables in a similar form as we did in previous sections.

Now, let us consider $\C\langle X_1,...,X_n\rangle$ as an algebra and define a coaction of the quantum semigroups $\B_s(n)$ to be a homomorphism $$\L'_n:\C\langle X_1,...,X_n\rangle\rightarrow\C\langle X_1,...,X_n\rangle\otimes\B_s(n)$$ by the following formulas:
$$\L'_n(1)=1\otimes I,\,\,\,\L'_n(X_i)=\sum\limits_{k=1}^n X_k\otimes \p u_{k,i}\p.$$
Then, we would have 
$$L_n(X_{i_1}\cdots X_{i_k})=\sum\limits_{j_1,...j_k=1}^n X_{j_1}\cdots X_{j_k}\otimes  \p u_{j_1,i_1}\p \cdots \p u_{j_n,i_n}\p$$
and $$(\L'_n\otimes id_{B_s(n)})\L'_n=(id_{\C_n} \otimes \Delta)\L'_n.$$
We will call $\L'_n$ the algebraic coaction of $\B_s(n)$ on $\C\langle X_1,...,X_n\rangle$.
The invariance condition is so strong that we can get our conclusion in some  finitely generated probability spaces.

\begin{proposition}
Let $(\A,\phi)$ be a $W^*$-probability space with a non-degenerated state $\phi$, Fixed $n\in\mathbb{N}$, let $(x_i)_{i=1,...,n}$ be a sequence of selfadjoint noncommutative random variables in $\A$. We say the joint distribution of $(x_i)_{i=1,...,n}$ is invariant under the algebraic coaction $\L'_n $ of $\B_s(n)$ if 
$$\mu_{x_1,...,x_n}(p)\p=\mu_{x_1,...,x_n}\otimes id_{\B_s(n)}(\L'_n(p)),$$
for all $p\in\C\langle X_1,...,X_n\rangle$, where $\mu_{x_1,...,x_n}$ is the joint distribution of  $(x_i)_{i=1,...,n}$. If  is the WOT closure of the unital algebra generated by $(x_i)_{i=1,...,n}$, then
the joint distribution of $(x_i)_{i=1,...,n}$ is invariant under the algebraic coaction $\L'_n$ of $\B_s(n)$ is equivalent to  $x_1=x_2=\cdots=x_n$.
\end{proposition}
\begin{proof}
Suppose $x_1=x_2=\cdots=x_n$. Let $p=X_{i_1}\cdots X_{i_m}\in\C\langle X_1,...,X_n\rangle$, then we have 
$$
\begin{array}{rcl}
& &\mu_{x_1,...,x_n}\otimes id_{\B_s(n)}(\L'_n(X_{i_1}\cdots X_{i_m}))\\
&=&\mu_{x_1,...,x_n}\otimes id_{\B_s(n)}(\sum\limits_{j_1,...j_m=1}^n X_{j_1}\cdots X_{j_m}\otimes \p u_{j_1,i_1}\p u_{j_2,i_2}\p\cdots\p u_{j_m,i_m} \p)\\
&=& \sum\limits_{j_1,...j_m=1}^n \phi(x_{j_1}\cdots x_{j_m}) \p u_{j_1,i_1}\p u_{j_2,i_2}\p\cdots\p u_{j_m,i_m} \p  \\
&=&  \sum\limits_{j_1,...j_m=1}^n \phi(x_1^m)\p u_{j_1,i_1}\p u_{j_2,i_2}\p\cdots\p u_{j_m,i_m} \p\\
&=&\phi(x_1^m)\p\\
&=&\mu_{x_1,...,x_n}(X_{i_1}\cdots X_{i_m})\p.\\
\end{array}
$$
Since $p$ is arbitrary, we proved $\Leftarrow$.\\
Suppose the joint distribution of $(x_i)_{i=1,...,n}$ is invariant under the algebraic coaction $\L'_n $. Let $\{v_1,...,v_{2n}\}$ be orthonormal basis of the standard $2n-$dimensional Hilbert space $\mathbb{C}^{2n}$ and denote $v_k=v_{k+2n}$ for all $k\in\mathbb{Z}$. Let $$P_{i,j}=P_{v_{2(i-j)+1}+v_{2(j-i)+2}}$$
and $$P=P_{v_1+v_2+\cdots+v_{2n}},$$
where $P_v$ is the orthogonal projection onto the one dimensional subspace generated by the  vector $v$ . By lemma \ref{4.1}, we have a representation $\pi$ of $\B_s(n)$ on $\C^{2n}$ defined by the following formulas:
$$\pi(\p)=P,\,\,\,\,\, \pi(\p u_{i_1,j_1}\cdots u_{i_k,j_k}\p)=PP_{i_1,j_1}\cdots P_{i_k,j_k}P$$
for all $i_1,j_1,...,i_k,j_k\in \{1,...,n\}$ and $k\in \mathbb{N}$.
In particular, we have 
$$\pi(\p u_{i,j}\p)=PP_{i,j}P=\frac{1}{n}P.$$
Let $\pi$ act on the invariance condition, we get 
$$
\begin{array}{rcl}
\phi(x_{i_1}\cdots x_k)P&=&\pi(\mu_{x_1,...x_n}(X_{i_1}\cdots X_{i_k})\p)\\
&=&\pi(\sum\limits_{j_1,...,j_k=1}^n \mu_{x_1,...x_n}\otimes id_{\B_s(n)}(X_{j_1}\cdots X_{j_k}\otimes \p u_{j_1,i_1}\p \cdots\p u_{j_k,i_k}\p))\\
&=&\sum\limits_{j_1,...,j_k=1}^n \phi(x_{j_1}\cdots x_{j_k})\pi(\p u_{j_1,i_1}\p \cdots\p u_{j_k,i_k}\p))\\
&=&\sum\limits_{j_1,...,j_k=1}^n \phi(x_{j_1}\cdots x_{j_k})\frac{1}{n^k}P,\\
\end{array}
$$
for all $i_1,i_k\in\{1,...n\}$. It implies that
$$\phi(x_{i_1}\cdots x_k)= \frac{1}{n^k}\sum\limits_{j_1,...,j_k=1}^n \phi(x_{j_1}\cdots x_{j_k}).$$

Therefore, two mixed moments are the same if their  degree are the same. Given two monomials $a=x_{s_1}\cdots x_{s_{k_1}}$ and $b=x_{s_1}\cdots x_{s_{k_2}}$ then
$$\phi(a(x_i-x_j)(x_i-x_j)^*b)=\phi(a(x_i-x_j)^2b)=\phi(ax_ix_ib)-\phi(ax_ix_jb)-\phi(ax_ix_jb)+\phi(ax_ix_ib)=0,$$
the last equation is true because all the monomials have the same degree. By some linear combinations, we have 
$$\phi(a(x_i-x_j)(x_i-x_j)^*b)=0,$$
for all $a,b\in\A_{[n]}$, where $\A_{[n]}$ is the unital algebra generated by $x_1,...x_n$, thus 
$$x_i=x_j,$$
for all $i\neq j$.
\end{proof}

In the end of this section, we study a coaction of the  quantum semigroups $B_s(n)$ on the joint distribution of a sequence of noncommutative random variables.
We can define a  coaction $$L_n:\C\langle X_1,...,X_n\rangle\rightarrow\C\langle X_1,...,X_n\rangle\otimes B_s(n)$$ of $B_s(n)$ on the algebra of noncommutative polynomials $\C\langle X_1,...,X_n\rangle$  by the following formulas:
$$L_n(1)=1\otimes I,\,\,\,\, L_n(X_i)=\sum\limits_{k=1}^n X_k\otimes u_{k,i},$$
where $1$ is the identity of $\C\langle X_1,...,X_n\rangle$ and $I$ is the identity of $B_s(n)$.
According to the definition of $L_n$, we have 
$$L_n(X_{i_1}\cdots X_{i_k})=\sum\limits_{j_1,...j_k=1}^nX_{j_1}\cdots X_{j_k}\otimes  u_{j_1,i_1}\cdots u_{j_n,i_n}.$$
One can easily check $$(L_n\otimes id_{B_s(n)})L_n=(id_{\C_n} \otimes \Delta)L_n, $$
where $id_{B_s(n)}$ and $id_{\C_n}$ are identity map of $B_s(n)$ and $\C\langle X_1,...,X_n\rangle$.\\
Then, we have  
\begin{proposition}
Let $(\A,\phi)$ be a probability space and $(x_i)_{i\in\mathbb{N}}$. The joint distribution is invariant under the coactions $L_n$'s of the quantum semigroups $B_s(n)$ if
for all $n$, $p\in \C\langle X_1,...,X_n\rangle$. Let $\mu_{x_1,...,x_n}$ be the joint distribution of $x_1,...,x_n$, $I$ be the unit of $B_s(n)$ and $L_n$ is the coaction on $\C\langle X_1,...,X_n\rangle$, we have 
$$\mu_{x_1,...,x_n}(p)I=\mu_{x_1,...,x_n}\otimes id(L_n(p)).$$
If the joint distribution of $(x_i)_{i\in\mathbb{N}}$ is invariant under the coactions $L_n$'s of the quantum semigroups $B_s(n)$'s, 
then $\phi(x_{i_1}x_{i_2}\cdots x_{i_k})=0$ for all $i_1,...,i_k\in \mathbb{N}$ and $k\in\mathbb{N}$.
\end{proposition}
\begin{proof}
Let $k$ be a positive integer, $i_1,...,i_k\in\mathbb{N}$ and $N=\max\{i_1,...,i_k\}$. Take a trivial representation $\pi$ of $B_s(N)$ on a $1-$dimensional space $V$ defined by the following formulas:
$$\pi(I)=1,\,\,\,\text{and}\,\,\,\pi(u_{i,j})=\pi(\p)=0,$$ where $1$ is the identity of $V$. By the universality of $B_s(n)$, $\pi$ is well defined. According to the invariance condition, we have
$$\pi(\mu_{x_1,...,x_N}(p)I)=\mu_{x_1,...,x_N}\otimes \pi(L_n(p)),$$
for all $p\in\C\langle X_1,...,X_n\rangle$.
Let $p=X_{i_1}\cdots X_{i_k}$, we get 
$$\phi(x_{i_1}x_{i_2}\cdots x_{i_k})1=\sum\limits_{j_1,...j_k=1}^n \phi(x_{j_1}\cdots x_{j_k})\pi( u_{j_1,i_1}\cdots u_{j_n,i_n})=0,$$ 
which completes the proof.
\end{proof}

\noindent{\bf Acknowledgement} I would like to thank Dan-Virgil Voiculescu for his suggestions and supports while working on this project. I would also like to thank Claus K\"{o}stler and Francesco Fidaleo for helpful discussions and Hasebe for pointing out that the original framework does not contain pairs of  non-trivial boolean independent random variables. While working on this paper, the author was supported  in part by funds from NSF grant DMS-1001881 and DMS-1301727. The author was benefited a lot from attending the Focus Program on Noncommutative Distributions in Free Probability Theory at the Fields Institute at the University of Toronto in July of 2013. His travel expenses for this conference were funded by NSF grant DMS-1302713.

Department of Mathematics\\
University of California at Berkeley\\
Berkeley, CA 94720, USA\\
E-MAIL: weihualiu@math.berkeley.edu\\

\end{document}